\documentclass{article}
\usepackage{amscd}
\usepackage{times}
\usepackage{latexsym}
\usepackage{graphicx}
\usepackage{amsmath}
\usepackage{amssymb}
\usepackage{dsfont}
\usepackage{mathrsfs}
\usepackage{epsfig}
\usepackage{graphicx,graphics}
\usepackage{pstricks}
\usepackage{pst-all}

\newtheorem{theorem}{Theorem}[section]
\newtheorem{definition}[theorem]{Definition}
\newtheorem{lemma}[theorem]{Lemma}
\newtheorem{proposition}[theorem]{Proposition}

\newtheorem{conj}{Conjecture}
\newcommand{\dist}{\mathop{\mathrm{dist}}}
\newenvironment{proof}[1][Proof]{\noindent\textbf{Proof:} }{\hfill\rule{3mm}{3mm}}

\makeatletter \@addtoreset{equation}{section} \makeatletter

\begin{document}

\title {\Large\bf Dimensions of Biquadratic Spline Spaces over T-meshes}
\author{Jiansong Deng  \quad
Falai Chen \quad Liangbing Jin \\
\small Department of Mathematics \\
\small University of Science and Technology of China\\
\small Hefei, Anhui, P. R. of China\\
\small Email: \texttt{dengjs|chenfl@ustc.edu.cn}}
\date{}
\maketitle

\emph{\textbf{Abstract} This paper discusses the dimensions of the
spline spaces over T-meshes with lower degree. Two new concepts are
proposed: extension of T-meshes and spline spaces with homogeneous
boundary conditions. In the dimension analysis, the key strategy is
linear space embedding with the operator of mixed partial
derivative. The dimension of the original space equals the
difference between the dimension of the image space and the rank of
the constraints which ensuring any element in the image space has a
preimage in the original space. Then the dimension formula and basis
function construction of bilinear spline spaces of smoothness order
zero over T-meshes are discussed in detail, and a dimension lower
bound of biquadratic spline spaces over general T-meshes is
provided. Furthermore, using level structure of hierarchical
T-meshes, a dimension formula of biquadratic spline space over
hierarchical T-meshes are proved. A topological explantation of the
dimension formula is shown as well.}

\section{\label{Introduction}Introduction}

A T-mesh is a rectangular grid that allows T-junctions. T-splines, a
type of point-based splines defined over a T-mesh, were proposed by
T. W. Sederberg in \cite{Sed03,Sed04}, and have become an important
tool in geometric modeling and surface reconstruction. For
T-splines, T-mesh plays two roles: defining its parametric domain
decomposition fashion and representing the topological structure of
the control net of a T-spline surface.

According to its definition, a T-spline is a piecewise polynomial,
instead of a single polynomial, within a cell of the T-mesh. This is
incompatible with the standard defining fashion of classic splines.
Recall that, given a spline knot sequence, a univariate spline can
be defined, which is a single polynomial between any two neighboring
knots. Hence in \cite{Deng06}, some of the present authors with some
other coauthors introduced the concept of spline spaces over
T-meshes, where every function in the space is exactly a polynomial
within each cell of the T-mesh. A dimension formula can be proved
with the B-net method \cite{Deng06} and the smoothing cofactor
method \cite{huang} for the spline space
$\mathbf{S}(m,n,\alpha,\beta,\mathscr{T})$ as $m\geqslant 2\alpha+1$
and $n\geqslant 2\beta+1$. Then in \cite{Deng07} we provided an
approach to define the basis functions of the bicubic splines with
smoothness order one. The applications of the basis functions in
surface fitting is explored as well.

According the dimension formula in \cite{Deng06}, one can obtain the
specified dimension formulae for some spline spaces with low degree
as $\mathbf{S}(1,1,0,0,\mathscr{T})$,
$\mathbf{S}(2,2,0,0,\mathscr{T})$,
$\mathbf{S}(3,3,0,0,\mathscr{T})$, and
$\mathbf{S}(3,3,1,1,\mathscr{T})$. Furthermore, with similar
approaches in \cite{Deng07,Li}, one can construct their basis
functions with some ``good'' properties, say, compact support,
nonnegativity, and forming a partition of unity. In order to achieve
high order smoothness with as low as possible degree, we expect to
obtain the dimension formulae and basis functions construction for
the spline spaces $\mathbf{S}(m,n,m-1,n-1,\mathscr{T})$. Especially,
the most interesting ones are those for
$\mathbf{S}(2,2,1,1,\mathscr{T})$ and
$\mathbf{S}(3,3,2,2,\mathscr{T})$. In the paper, we only focus on
the dimension of the former space.

In the following, we first introduce two concepts: the spline space
with homogenous boundary conditions (HBC) and the extended T-meshes
associated with some spline space. As a foundation of the later
analysis, we discuss in detail the dimension formula and basis
functions construction for the space
$\mathbf{S}(1,1,0,0,\mathscr{T})$. In \cite{Deng06} we have shown
that the dimension of a bilinear spline space is the sum of the
numbers of crossing vertices and boundary vertices in the given
T-mesh. However, the proof proposed here shares the same fashion as
in the analysis of a lower bound of dimensions of biquadratic spline
spaces, and avoids the problem of recycling dependence of
T-vertices.

An important technique in the dimension analysis is linear space
embedding by an operator of mixed partial derivative, which embeds
the space $\mathbf{S}(m,n,m-1,n-1,\mathscr{T})$ into the space
$\mathbf{S}(m-1,n-1,m-2,n-2,\mathscr{T})$. A necessary and
sufficient condition for describing any element in
$\mathbf{S}(m-1,n-1,m-2,n-2,\mathscr{T})$ is the image of an element
in $\mathbf{S}(m,n,m-1,n-1,\mathscr{T})$ by the operator. In the
paper, we just discuss cases of $m=n=1$ or 2. With the method, a
lower bound of dimensions of $\mathbf{S}(2,2,1,1,\mathscr{T})$ is
proved. Finally, by making use of the level structure of
hierarchical T-meshes, a dimension formula is provided for
biquadratic spline spaces over hierarchical T-meshes.

The paper is organized as follows. In Section 2, the spline spaces
over T-meshes is reviewed and two concepts are proposed: extension
of T-meshes and spline spaces with homogeneous boundary conditions.
Then the dimension formula and basis function construction of
bilinear spline spaces of smoothness order zero over T-meshes are
discussed in detail in Section 3, and a dimension lower bound of
biquadratic spline spaces over general T-meshes is provided in
Section. In Section 5, using level structure of hierarchical
T-meshes, a dimension formula of biquadratic spline space over
hierarchical T-meshes are proved. A topological explantation of the
dimension formula is shown as well. Section 6 concludes the paper
with some discussions.

\section{T-meshes and Spline Spaces}

A {\bf T-mesh} is basically a rectangular grid that allows
T-junctions.  As for details of T-meshes, please refer to
\cite{Deng06}. Here the T-meshes we discuss are regular, i.e., the
domaines occupied by the T-meshes are rectangles. We adopt the
compatible definitions of vertex, edge, cell with those in
\cite{Deng06}.

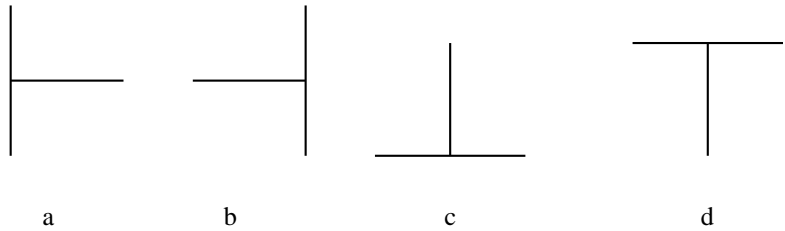
\begin{figure}[!ht]
\begin{center}
\begin{tabular}{cccc}
\begin{pspicture}(0,0)(2,3)
\psline(0.5,0.5)(0.5,2.5)\psline(0.5,1.5)(2,1.5)
\end{pspicture} &
\begin{pspicture}(0,0)(2,3)
\psline(2,0.5)(2,2.5)\psline(0.5,1.5)(2,1.5)
\end{pspicture} &
\begin{pspicture}(0,0)(3,2)
\psline(0.5,0.5)(2.5,0.5)\psline(1.5,0.5)(1.5,2)
\end{pspicture} &
\begin{pspicture}(0,0)(3,2)
\psline(0.5,2)(2.5,2)\psline(1.5,0.5)(1.5,2)
\end{pspicture}\\
a  & b & c & d
\end{tabular}
\caption{Horizontal T-vertices (a and b) and vertical T-vertices (c
and d) \label{TVertex}}
\end{center}
\end{figure}

T-vertices in a T-mesh can be classified into two types:
\textbf{horizontal T-vertices} and \textbf{vertical T-vertices}. The
T-vertices as shown in Figures \ref{TVertex}.a and b are horizontal
ones, and those as shown in Figures \ref{TVertex}.c and d are
vertical ones. In \cite{Deng06}, edges and c-edges are defined. Here
we introduce a new type of edges, l-edges. A
\textbf{horizontal/vertical l-edge} is a continuous line segment
which consists of some horizontal/vertical interior edges and whose
two endpoints are boundary vertices or horizontal/vertical
T-vertices. In other words, an l-edge is a longest possible line
segment in the mesh. The boundary of a regular T-mesh consists of
four l-edges, which are called boundary l-edges. The other l-edges
in the T-mesh are called interior l-edges. For example, in Figure
\ref{TExpand}, the given T-mesh $\mathscr{T}$ has three horizontal
interior l-edges and three vertical interior l-edges, respectively.

Given two series of real numbers $x_i$, $i=1,\ldots, m$, and $y_j$,
$j=1, \ldots, n$, where $x_i<x_{i+1}$ and $y_j < y_{j+1}$, a
rectangular grid can be formed with vertices $(x_i, y_j)$, $i=1,
\ldots, m$, $j=1, \ldots, n$. This grid is called a
\textbf{tensor-product mesh}, denoted by $(x_1, \ldots, x_m) \times
(y_1, \ldots, y_n)$, which is a special type of T-mesh. From a
regular T-mesh $\mathscr{T}$, a tensor-product mesh $\mathscr{T}^c$
can be constructed by extending all the interior l-edges to the
boundary. $\mathscr{T}^c$ is called the \textbf{associated
tensor-product mesh} with $\mathscr{T}$. See Figure \ref{TExpand}
for an example.
\begin{figure}[!htb]
\begin{center}
\psset{unit=0.75cm,linewidth=0.8pt}
\begin{tabular}{ccc}
\begin{pspicture}(0.5,0.5)(6,6)
\psline(1,1)(1,5)(5,5)(5,1)(1,1) \psline(1,2)(5,2)\psline(2,3)(4,3)
\psline(1,4)(5,4)\psline(2,1)(2,4)
\psline(3,2)(3,5)\psline(4,1)(4,5)
\end{pspicture} &
\begin{pspicture}(0.5,0.5)(6,6)
\psline(1,1)(1,5)(5,5)(5,1)(1,1) \psline(1,2)(5,2)\psline(1,3)(5,3)
\psline(1,4)(5,4)\psline(2,1)(2,5)
\psline(3,1)(3,5)\psline(4,1)(4,5)
\end{pspicture}&
\begin{pspicture}(0.5,0.5)(6,6)
\psline(0.6,0.6)(0.6,5.4)(5.4,5.4)(5.4,0.6)(0.6,0.6)
\psline(0.6,2)(5.4,2)\psline(2,3)(4,3)
\psline(0.6,4)(5.4,4)\psline(2,0.6)(2,4)
\psline(3,2)(3,5.4)\psline(4,0.6)(4,5.4)
\psline(0.6,0.8)(5.4,0.8)\psline(0.6,1)(5.4,1)
\psline(0.6,5)(5.4,5)\psline(0.6,5.2)(5.4,5.2)
\psline(0.8,0.6)(0.8,5.4)\psline(1,0.6)(1,5.4)
\psline(5,0.6)(5,5.4)\psline(5.2,0.6)(5.2,5.4)
\end{pspicture} \\
$\mathscr{T}$ &$\mathscr{T}^c$ & $\mathscr{T}^\varepsilon$
\end{tabular}
\caption{A T-mesh with its associated tensor-product mesh
$\mathscr{T}^c$ and its extension
$\mathscr{T}^{\varepsilon}$.\label{TExpand}}
\end{center}
\end{figure}
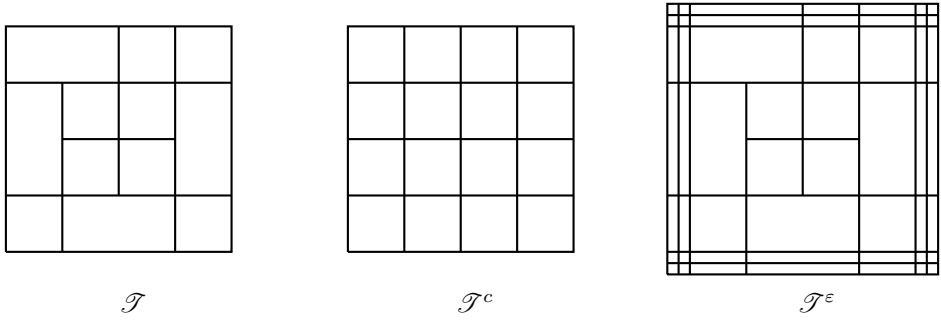

\subsection{Spline spaces over T-meshes}

Given a T-mesh $\mathscr{T}$, we use $\mathscr{F}$ to denote all the
cells in $\mathscr{T}$ and $\Omega$ to denote the region occupied by
all the cells in $\mathscr{T}$. In \cite{Deng06}, the following
spline space definition is proposed:
\begin{align}
\mathbf{S}(m,n,\alpha,\beta,\mathscr{T}):=\{f(x,y)\in
C^{\alpha,\beta}(\Omega): f(x,y)|_{\phi}\in \mathbb{P}_{mn},\forall
\phi\in\mathscr{F}\},
\end{align}
where $\mathbb{P}_{mn}$ is the space of all the polynomials with
bi-degree $(m,n)$, and $C^{\alpha,\beta}$ is the space consisting of
all the bivariate functions which are continuous in $\Omega$ with
order $\alpha$ along $x$ direction and with order $\beta$ along $y$
direction. It is obvious that
$\mathbf{S}(m,n,\alpha,\beta,\mathscr{T})$ is a linear space.

Now we introduce a new spline space over T-meshes with the following
defintion:
\begin{align}
\overline{\mathbf{S}}(m,n,\alpha,\beta,\mathscr{T}):=\{f(x,y)\in
C^{\alpha,\beta}(\mathds{R}^2): f(x,y)|_{\phi}\in
\mathbb{P}_{mn},\forall \phi\in\mathscr{F},\ \mathrm{ and }\
f|_{\mathds{R}^2\setminus\Omega}\equiv0\},
\end{align}
which is called a spline space over the given T-mesh $\mathscr{T}$
with \textbf{homogeneous boundary conditions (HBC)}.

\subsection{Extended T-meshes}

In order to discuss the dimension of the spline space
$\mathbf{S}(m,n,m-1,n-1,\mathscr{T})$, we extend the given regular
T-mesh $\mathscr{T}$ in the following fashion: Pickup a
tensor-product mesh $\mathscr{M}$ with $2(m+1)$ vertical lines and
$2(n+1)$ horizontal lines, such that the central rectangle of
$\mathscr{M}$ is the same size as $\Omega$, the region occupied by
$\mathscr{T}$; Then all the edges with an endpoint on the boundary
of $\mathscr{T}$ are extended to reach the boundary of
$\mathscr{M}$. The result mesh, denoted as
$\mathscr{T}^\varepsilon$, is called an extension of the original
T-mesh (or say, an extended T-mesh) associated with the present
spline space. Figure \ref{TExpand} provides an example, where the
T-mesh is extended associated with
$\mathbf{S}(2,2,1,1,\mathscr{T})$.

It should be noted that the extension is associated with a spline
space. Hence, associated with different spline spaces, one will
obtain different extended T-meshes. The following theorem shows
that, using an extension of T-mesh, the dimension analysis of
$\overline{\mathbf{S}}(2,2,1,1,\mathscr{T}^\varepsilon)$ is the same
as the dimension analysis of ${\mathbf{S}}(2,2,1,1,\mathscr{T})$.

\begin{theorem}
\label{thm2.1} Given a T-mesh $\mathscr{T}$, assume
$\mathscr{T}^{\varepsilon}$ is its extension associated with
$\mathbf{S}(2,2,1,1,\mathscr{T})$. Then
\begin{align} \label{eqn2.3}
\mathbf{S}(2,2,1,1,\mathscr{T})&=
                       \overline{\mathbf{S}}(2,2,1,1,\mathscr{T}^{\varepsilon})|_{\mathscr{T}},\\
\dim \mathbf{S}(2,2,1,1,\mathscr{T})&=
                       \dim \overline{\mathbf{S}}(2,2,1,1,\mathscr{T}^{\varepsilon}).
                       \label{eqn2.4}
\end{align}
\end{theorem}
\begin{proof}
We first prove Equation \eqref{eqn2.3}. Since
$\mathscr{T}^{\varepsilon}$ is an extension of $\mathscr{T}$, it
follows that
$$\overline{\mathbf{S}}(2,2,1,1,\mathscr{T}^{\varepsilon})|_{\mathscr{T}}
                                           \subset\mathbf{S}(2,2,1,1,\mathscr{T}).$$
In the following we will prove that, for any $
f\in\mathbf{S}(2,2,1,1,\mathscr{T})$, there exists
$\bar{f}\in\overline{\mathbf{S}}(2,2,1,1,
\mathscr{T}^{\varepsilon})$ such that $\bar{f}|_{\mathscr{T}}=f$.

\begin{figure}[!htb]
\begin{center}
\setlength{\unitlength}{1cm}
\begin{pspicture}(0.5,0.5)(6.5,5)
\psline(1,1)(6.5,1)\psline(1,2)(6.5,2) \psline(1,3)(6.5,3)
\psline(1,4)(3,4) \psline(1,1)(1,4.5)\psline(2,1)(2,4.5)
\psline(3,1)(3,4.5)\psline(4,1)(4,3)\psline(5,1)(5,3)\psline(6,1)(6,3)
\psdots[dotstyle=triangle](1,1)(1,1.5)(1,2)(1,2.5)(1,3)(1.5,1)(1.5,1.5)(1.5,2)(1.5,2.5)(1.5,3)(2,1)(2,1.5)(2.5,1)
(2.5,1.5)(3,1)(3,1.5)(3.5,1)(3.5,1.5)(4,1)(4,1.5)(4.5,1)(4.5,1.5)(5,1)(5,1.5)(5.5,1)(5.5,1.5)(6,1)(6,1.5)
\psdots(2.5,2.5)(2.5,3)(3,2.5)(3,3)(3.5,2.5)(3.5,3)(4,2.5)(4,3)(4.5,2.5)(4.5,3)(5,2.5)(5,3)(5.5,2.5)(5.5,3)(6,2.5)(6,3)
\psdots[dotstyle=square](2,2.5)(2,3)(2.5,2)(3,2)(3.5,2)(4,2)(4.5,2)(5,2)(5.5,2)(6,2)
\psdots[dotstyle=square*](2,2) \rput(4.5,4){$\mathscr{T}$}
\end{pspicture}
\caption{B\'ezier ordinates in an extended T-mesh\label{Th1}}
\end{center}
\end{figure}
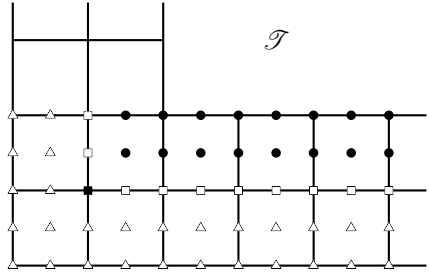

Here we define the function $\bar{f}$ over
$\mathscr{T}^{\varepsilon}$ by assigning its B\'ezier ordinates in
every cell of $\mathscr{T}^{\varepsilon}\backslash \mathscr{T}$.
Since $\bar{f} \in C^{1,1}$, and there does not exist any interior
T-vertices in $\mathscr{T}^{\varepsilon}\backslash \mathscr{T}$,
$\bar{f}$ shares the same ordinates along the common boundary
between two neighboring cells. (Therefore, these same ordinates
appears once in Figure \ref{Th1}, which illustrates the distribution
of the B\'ezier ordinates in the left-bottom part of the extended
region $\mathscr{T}^{\varepsilon}\backslash \mathscr{T}$.) According
to the specification of $f$ over $\mathscr{T}$, it follows that the
ordinates labeled with ``$\bullet$'' are determined. Since $\bar{f}$
meets the zero function along the boundary of the extended T-mesh
$\mathscr{T}^{\varepsilon}$ with order one, the ordinates labeled
with ``$\vartriangle$'' are determined as well. The rest ordinates
labeled with ``$\Box$'' can be determined by its corresponding
neighboring horizontal or vertical ordinates labeled with
``$\vartriangle$'' and ``$\bullet$''. Here the continuous condition
is the collinear of the corresponding three points. It should be
noted that the ordinates themselves determined in this fashion
ensures that they define a continuous spline automatically.

After determining the former ordinates, the ordinate labeled with
``{\tiny $\blacksquare$}'' can be determined as follows: It follows
that the eight ordinates neighboring the ordinate ``{\tiny
$\blacksquare$}'' define a bilinear function. When the ordinate
corresponding to ``{\tiny $\blacksquare$}'' also lies on the
bilinear function, the continuous of order one along two directions
can be guaranteed. In fact, when dealing with the determination of
all the ordinates in the extension region around $\mathscr{T}$, if
we meet cases of recycling determinations, we can address the
ordinates using this method of bilinear functions construction.

Hence we obtain the specification of the function $\bar{f}$ in
$\mathbf{S}(2,2,1,1,\mathscr{T}^{\varepsilon})$, which satisfies
$\bar{f}|_{\mathscr{T}}=f$. According to the former analysis, it
follows that
$$\mathbf{S}(2,2,1,1,\mathscr{T})=
\overline{\mathbf{S}}(2,2,1,1,\mathscr{T}^{\varepsilon})|_{\mathscr{T}}.$$

In order to prove Equation \eqref{eqn2.4}, we will prove that, for
any nonzero function
 $\bar{f}\in\overline{\mathbf{S}}(2,2,1,1,\mathscr{T}^{\varepsilon})$,
it follows that $\bar{f}|_{\mathscr{T}}\not \equiv 0$. Suppose
$\bar{f}|_{\mathscr{T}} \equiv 0$, i.e., there exists $p=(x_0,y_0)
\in \mathscr{T}^{\varepsilon}\backslash \mathscr{T}$ such that
$\bar{f}(x_0,y_0) \neq 0$. And assume $\Omega = (x_l, x_r) \times
(y_b,y_t)$. If $x_l \leqslant x_0 \leqslant x_r$, $y_0< y_b$, then
$\bar{f}(x_0,y)$, $y \leqslant y_b$ is a quadratic spline, whose
support has just three breakpoints. This contradicts with the fact
about the minimal support of quadratic splines (The support of a
nonzero quadratic spline should have at least four breakpoints).
Similarly, we can show that $p$ is outside of the regions $x_l
\leqslant x_0 \leqslant x_r$, $y_0> y_t$; $y_l \leqslant y_0
\leqslant y_r$, $x_0< x_l$ and $y_b \leqslant y_0 \leqslant y_t$,
$x_0> x_r$. Therefore, without loss of generality, we assume
$x_0<x_r$, $y_0<y_b$. But this also results in a nonzero quadratic
spline with three breakpoints in its support. Hence we obtain
$\bar{f}|_{\mathscr{T}}\not \equiv 0$, and has proved that
$$\dim\mathbf{S}(m,n,\alpha,\beta,\mathscr{T})=
\dim\overline{\mathbf{S}}(m,n,\alpha,\beta,\mathscr{T}^{\varepsilon}).$$
This completes the proof of the theorem.
\end{proof}

According to Theorem \ref{thm2.1}, we will consider only the
biquadratic spline spaces over T-meshes with HBC.

\noindent \textbf{Remark}: A similar analysis can show that, for any
$m$ and $n$, it follows that
\begin{align}
\mathbf{S}(m,n,m-1,n-1,\mathscr{T})&=
                       \overline{\mathbf{S}}(m,n,m-1,n-1,
                       \mathscr{T}^{\varepsilon})|_{\mathscr{T}},\\
\dim \mathbf{S}(m,n,m-1,n-1,\mathscr{T})&=
                       \dim
                       \overline{\mathbf{S}}(m,n,m-1,n-1,\mathscr{T}^{\varepsilon}),
\end{align}
where $\mathscr{T}^{\varepsilon}$ is an extension of $\mathscr{T}$
associated with $\mathbf{S}(m,n,m-1,n-1,\mathscr{T})$. Hence one can
replace the dimension discussion for
$\mathbf{S}(m,n,m-1,n-1,\mathscr{T})$ with that for
$\overline{\mathbf{S}}(m,n,m-1,n-1,\mathscr{T}^{\varepsilon})$.

\section{Dimensions and Basis Functions of $\overline{\mathbf{S}}(1,1,0,0,\mathscr{T})$}

In \cite{Deng06} we have proved the following dimension formula of
the spline space $\mathbf{S}(m,n,\alpha,\beta,\mathscr{T})$ over
T-mesh:
\begin{align}
\dim \mathbf{S}(m,n,\alpha,\beta,\mathscr{T})=F(m+1)(n+1)-E_h(m+1)(\beta+1)\nonumber\\
   -E_v(n+1)(\alpha+1)+V(\alpha+1)(\beta+1),
\end{align}
where $m\geqslant 2\alpha+1$, $n\geqslant 2\beta+1$, $F$ is the
number of all the cells in $\mathscr{T}$, $E_h$ and $E_v$ the
numbers of horizontal and vertical interior edges, respectively, and
$V$ the number of interior vertices (including crossing and
T-vertices). Specially, as $m=n=1$ and $\alpha=\beta=0$, it follows
that
\begin{align}
\dim \mathbf{S}(1,1,0,0,\mathscr{T})=V^+ + V^b,
\end{align}
where $V^+$ is the number of crossing vertices, and $V^b$ the number
of boundary vertices. Using an extension of the T-meshes associated
with $\mathbf{S}(1,1,0,0,\mathscr{T})$, it follows that
\begin{align}
\dim \overline{\mathbf{S}}(1,1,0,0,\mathscr{T}^{\varepsilon}) = \dim
\mathbf{S}(1,1,0,0,\mathscr{T}) =V^+_{\varepsilon},
\end{align}
where $V^+_{\varepsilon}$ is the number of the crossing vertices in
the extended T-mesh $\mathscr{T}^{\varepsilon}$, which equals the
sum of the numbers of crossing vertices and boundary vertices in
$\mathscr{T}$.

Now we will prove that the former formula holds for a general
regular T-mesh (not just an extended T-mesh). The method taken in
the proof is similar with the method proposed in the next section
for proving the dimension properties of
$\overline{\mathbf{S}}(2,2,1,1,\mathscr{T})$. This method solves the
problem of recycling dependence of T-vertices which happens in the
proof proposed in \cite{Deng06}.

Now we introduce some notation and lemmas for the proof. In the
given T-mesh $\mathscr{T}$, let $E$ denote the number of interior
l-edges, and $V^+$ the number of crossing vertices. Define
$\overline{\mathbf{S}}(0,0,-1,-1,\mathscr{T})$ to be the space
consisting of functions which are a constant in each cell of the
T-mesh $\mathscr{T}$, and have no smoothness requirement between
neighboring cells. It is obvious that its dimension is the number of
all the cells in the T-mesh.

\subsection{Operator of mixed partial derivative}

The operator of mixed partial derivative is introduced as follows:
\begin{equation}\label{mpd} \mathcal{D}:=\frac{\partial ^2}{\partial
x
\partial y}: \overline{\mathbf{S}}(m,n,m-1,n-1,\mathscr{T}) \to
\overline{\mathbf{S}}(m-1,n-1,m-2,n-2,\mathscr{T}),\end{equation}
where $m,n\geqslant 1$. Since the function satisfies HBC, the
operator is a one-to-one, but not onto mapping. For example, for any
nonzero $g \in \overline{\mathbf{S}}(m,n,m-1,n-1,\mathscr{T})$, it
follows that $\mathcal{D}(g)$ must reach positive values in some
parts and negative values in some other parts. Hence a nonnegative
function in $\overline{\mathbf{S}}(m-1,n-1,m-2,n-2,\mathscr{T})$ has
no pre-image under the operator $\mathcal{D}$.

Define \begin{equation}\label{ig} \mathcal{I}(g)(x,y) =
\int_{-\infty}^x \int_{-\infty}^y g(s,t)dsdt.
\end{equation}
It follows that, for any $f \in
\overline{\mathbf{S}}(m,n,m-1,n-1,\mathscr{T})$, $\mathcal{I}\circ
\mathcal{D}(f) = f$. Hence $\mathcal{I}$ can be considered to be the
inverse operator of $\mathcal{D}$.

It follows that, for any $g \in
\overline{\mathbf{S}}(m-1,n-1,m-2,n-2,\mathscr{T})$, $
\mathcal{I}(g) $ is a piecewise polynomial of degree $(m,n)$ with
smoothness $C^{m-1,n-1}$. With respect to functions in
$\overline{\mathbf{S}}(m,n,m-1,n-1,\mathscr{T})$,
$\partial^{m}\mathcal{I}(g)/\partial x^m$ or
$\partial^{n}\mathcal{I}(g)/\partial y^n$ may be discontinuous
inside some cells of $\mathcal{T}$. But the discontinuity must
happen on the extension of some l-edges of $\mathcal{T}$. Let
$$ d_{m,n} = \dim \overline{\mathbf{S}}(m,n,m-1,n-1,\mathscr{T}).$$
If we have known $d_{m-1,n-1}$, i.e., the dimension of
$\overline{\mathbf{S}}(m-1,n-1,m-2,n-2,\mathscr{T})$, and the number
$r_{m-1,n-1}$ of linear-independent constraints ensuring
$\mathcal{I}(g) \in \overline{\mathbf{S}}(m,n,m-1,n-1,\mathscr{T})$
for any $g \in \overline{\mathbf{S}}(m-1,n-1,m-2,n-2,\mathscr{T})$,
then it follows that \begin{equation}\label{df} d_{m,n} =
d_{m-1,n-1} -r_{m-1,n-1}. \end{equation}
On the other hand, if there
are $r'_{m-1,n-1}$ constraints, which may be linear dependent,
proposed for ensuring $\mathcal{I}(g) \in
\overline{\mathbf{S}}(m,n,m-1,n-1,\mathscr{T})$ for any $g \in
\overline{\mathbf{S}}(m-1,n-1,m-2,n-2,\mathscr{T})$, then it follows
that \begin{equation} \label{lb} d_{m,n} \geqslant d_{m-1,n-1} -
r'_{m-1,n-1}.
\end{equation}
Equations \eqref{df} and \eqref{lb} are used to prove the dimension
formulae of bilinear and biquadratic spline spaces over T-meshes in
the rest of the paper.

\subsection{Some lemmas}
As $m=n=1$, the following lemma proposes the constraints ensuring
$\mathcal{I}(g) \in \overline{\mathbf{S}}(1,1,0,0,\mathscr{T})$ for
any $ g \in \overline{\mathbf{S}}(0, 0, -1, -1, \mathscr{T})$.

\begin{lemma}
\label{lemm3.1} Given a regular T-mesh $\mathscr{T}$, let $g \in
\overline{\mathbf{S}}(0,0,-1,-1,\mathscr{T})$. Then
$$\mathcal{I}(g)(x,y)
\in\overline{\mathbf{S}}(1,1,0,0,\mathscr{T})$$ if and only if the
following two sets of conditions are satisfied simultaneously:
\begin{enumerate}
\item For any horizontal l-edge $l^h$,
\begin{equation} \label{int11} \int_{x_0}^{x_1} g(s, y_0-) ds = \int_{x_0}^{x_1} g(s, y_0+)
ds,\end{equation} where the two end-points of $l^h$ are with
coordinates $(x_0,y_0)$ and $(x_1, y_0)$;
\item For any vertical l-edge $l^v$,
\begin{equation} \label{int22}\int_{y_0}^{y_1} g(x_0-, t) dt = \int_{y_0}^{y_1} g(x_0+, t)
dt,\end{equation} where the two end-points of $l^v$ are with
coordinates $(x_0,y_0)$ and $(x_0, y_1)$.
\end{enumerate}
\end{lemma}

\begin{proof}
For any horizontal l-edge $l^h$, its right end-point $(x_1,y_0)$ is
a T-vertex or a boundary vertex. Suppose the edge through the vertex
$(x_1,y_0)$ and perpendicular with $l^h$ is $e^v$ (as shown in
Figure \ref{HL1}). If $\mathcal{I}(g)(x,y) \in
\overline{\mathbf{S}}(1,1,0,0,\mathscr{T})$, then
$\mathcal{I}(g)|_{e^v}$ is a linear polynomial in a neighborhood of
$(x_1, y_0)$. Hence it follows that
\begin{figure}[!htb]
\begin{center}
\setlength{\unitlength}{1cm}
\begin{pspicture}(1,1)(8,3)
\psline(1,2)(6.4,2)\psline[linestyle=dotted](6.4,2)(6.9,2)\psline(6.9,2)(7,2)\psline(1,1)(1,3)
\psline(7,1)(7,3) \psline(2,2)(2,3)\psline(4,1)(4,3)
\psline(5.5,1)(5.5,2) \rput[l](7.1,2){$(x_{1},y_0)$}
\rput(1.5,2.5){$\alpha_{h_1}^u$}\rput(3,2.5){$\alpha_{h_2}^u$}\rput(5,2.5){$\alpha_{h_3}^u$}
\rput(2.5,1.5){$\alpha_{h_1}^d$}\rput(4.5,1.5){$\alpha_{h_2}^d$}\rput(6,1.5){$\alpha_{h_3}^d$}
\rput(6.7,1.5){$\cdots$}\rput(6.7,2.5){$\cdots$}\rput(7.2,3){$e^v$}
\rput(0.5,2){$l^h$}
\end{pspicture}
\caption{A horizontal l-edge (1)\label{HL1}}
\end{center}
\end{figure}
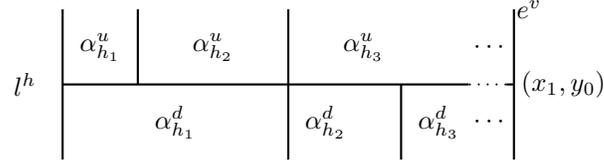
\begin{equation} \label{eqn3.71}
\frac{\partial}{\partial y} \mathcal{I}(g)(x_1, y_0-) =
\frac{\partial}{\partial y} \mathcal{I}(g)(x_1, y_0+).\end{equation}
According to the definition of $\mathcal{I}$ in Equation \eqref{ig},
Equation \eqref{eqn3.71} can be rewritten into
\begin{equation} \label{int1} \int_{-\infty}^{x_1} g(s,y_0-) ds =
\int_{-\infty}^{x_1} g(s,y_0+) ds.\end{equation} Extend $l^h$ to
reach the left boundary of the T-mesh. If there are not other
horizontal l-edges on the extension edge, then, in every cell that
intersects the extension edge, $g$ is a pure linear polynomial. Then
$g(s,y_0-) = g(s,y_0+)$ as $s \leqslant x_0$, and
\begin{equation} \label{int2} \int_{-\infty}^{x_0} g(s,y_0-) ds =
\int_{-\infty}^{x_0} g(s,y_0+) ds. \end{equation}  According to
Equations \eqref{int1} and \eqref{int2}, it follows that the
equation \eqref{int11} in the first condition holds for any
horizontal l-edge.

It there exist some other horizontal l-edges on the extension edge,
then we run through these horizontal l-edges from left to right. For
the first one, we can prove that Equation \eqref{int11} holds with
the former approach. Then we consider the other l-edges one by one,
and prove that the corresponding equation \eqref{int2} holds as
well, and then Equation \eqref{int11} holds. Hence we finish to
prove Equation \eqref{int11} holds for every horizontal l-edge.

For any vertical l-edge, we can prove that the corresponding
equation \eqref{int22} holds in a similar way. Hence the necessity
of the lemma is proved.

Now we will prove the sufficiency of the lemma.
%
Suppose $g \in \overline{\mathbf{S}}(0,0,-1,-1,\mathscr{T})$
satisfies the two sets of conditions in the lemma, but
$\mathcal{I}(g) \notin \overline{\mathbf{S}}(1,1,0,0,\mathscr{T})$,
i.e., there exists a cell where $\partial \mathcal{I}(g)/\partial x$
or $\partial \mathcal{I}(g)/\partial y$ has at least a discontinuous
point. Define two sets consisting of the cells of $\mathscr{T}$ as
follows. Let $\mathcal{B}_x$ denote all the cells where $\partial
\mathcal{I}(g)/\partial x$ has at least a discontinuous point; Let
$\mathcal{B}_y$ denote all the cells where $\partial
\mathcal{I}(g)/\partial y$ has at least a discontinuous point. Then
according to the assumption, $\mathcal{B}_x \cup \mathcal{B}_y$ is
not empty. Without loss of generality, we assume $\mathcal{B}_y$ is
not empty. Consider a cell in $\mathcal{B}_y$ whose left-bottom
corner has the minimal $y$ coordinate in $\mathcal{B}_y$. If there
exist more than one such cell, select one from them such that its
left-bottom corner has the minimal $x$ coordinate. Then now we have
selected a unique cell $c$. In $c$, $\partial
\mathcal{I}(g)/\partial y$ has at least one discontinuous point
$(\bar{x},\bar{y})$. Hence it follows that
\begin{equation}\label{eqnn} \int_{-\infty}^{\bar{x}}
g(s,\bar{y}-)ds \neq \int_{-\infty}^{\bar{x}} g(s,\bar{y}+)ds.
\end{equation}
Consider the horizontal straight line through $(\bar{x},\bar{y})$.
On the straight line if there does not exist an l-edge of
$\mathscr{T}$ on the left side of $(\bar{x},\bar{y})$, then Equation
\eqref{eqnn} fails to hold. If there exists at least one l-edge on
the left side of $(\bar{x},\bar{y})$, then according to the first
condition, a contradiction arises as well. Hence $\mathcal{B}_y$ is
empty. In the same fashion, it follows that $\mathcal{B}_x$ is empty
as well. This contradicts with $\mathcal{B}_x \cup \mathcal{B}_y$
non-empty. Therefore we finish the proof of the sufficiency.
\end{proof}

There are $E+4$ conditions in Lemma \ref{lemm3.1}. The following
lemma states that these $E+4$ conditions are equivalent with $E+2$
conditions in another form.

\begin{lemma}
\label{lemm3.2} Given a regular T-mesh $\mathscr{T}$, the occupied
rectangle by $\mathscr{T}$ is $(x_l,x_r) \times (y_b, y_t)$, where
the different $y$ coordinates of horizontal l-edges are $y_0 < y_1 <
\ldots < y_n$. Suppose $g \in
\overline{\mathbf{S}}(0,0,-1,-1,\mathscr{T})$. Then
\begin{equation}
\label{eqn3.4} \int_{x_l}^{x_r} g(s, y_i-) ds = \int_{x_l}^{x_r}
g(s, y_i+) ds, \quad i=0,1,\ldots,n \end{equation} is equivalent
with \begin{equation} \label{eqn3.5} \int_{x_l}^{x_r} g(s,y)ds = 0,
\quad y \in (y_i, y_{i+1}), \quad i = 0, \ldots, n-1.
\end{equation}
It should be noted that, since $g$ is a piecewise constant, $
\int_{x_l}^{x_r} g(s,y)ds$, $ y \in (y_i, y_{i+1})$, is a constant
independent on $y$.

A similar statement can be made for vertical l-edges.
\end{lemma}

\begin{proof}
Since $g$ is a piecewise constant in $\mathscr{T}$, and zero out of
$\mathscr{T}$, it follows that
\begin{align*} &
\int_{x_l}^{x_r} g(s, y_0-)ds
= \int_{x_l}^{x_r} g(s,y) ds = 0, \quad y < y_0, \\
& \int_{x_l}^{x_r} g(s, y_0+)ds = \int_{x_l}^{x_r} g(s,y) ds, \quad
y_0< y < y_1.
\end{align*}
Hence according to Equation \eqref{eqn3.4} as $i=0$, we have
$$\int_{x_l}^{x_r} g(s,y)ds = 0,
\quad y \in (y_0, y_{1}).$$ This means that Equation \eqref{eqn3.5}
holds as $i=0$. Then we consider cases as $i=1,2, \ldots, n-1$.
Similar discussions guarantee Equation \eqref{eqn3.5} holds as
$i=1,2, \ldots, n-1$.

On the other hand, for any $i=0,1,2, \ldots, n-1$, if all the
quations in \eqref{eqn3.5} hold, it is easy to prove that all the
equations in \eqref{eqn3.4} hold.
\end{proof}

Recall that, for a given T-mesh $\mathscr{T}$, an associated tensor
product mesh $\mathscr{T}^c$ can be obtained by extending all the
interior l-edges to the boundary of $\mathscr{T}$. Suppose there are
$E'$ interior l-edges in $\mathscr{T}^c$ and $E$ interior l-edges in
$\mathscr{T}$, where $E' \leqslant E$ (since it is possible that
more than one l-edges in $\mathscr{T}$ on the same l-edge in
$\mathscr{T}^c$). Select any horizontal l-edge $l$ from
$\mathscr{T}^c$, and suppose the horizontal l-edges $l_1, \ldots,
l_k$ in $\mathscr{T}$ lie on $l$. Then it follows that the
constraints
$$ \int_{l_i} g(s, \bar{y}-)ds
= \int_{l_i} g(s, \bar{y}+)ds, \quad i=1, \ldots, k$$ is equivalent
with the constraints
$$ \int_{l_i} g(s, \bar{y}-)ds = \int_{l_i} g(s, \bar{y}+)ds, \quad
i=1,\ldots, k-1, \quad  \int_{l} g(s, \bar{y}-)ds = \int_{l} g(s,
\bar{y}+)ds, $$ where $\bar{y}$ is the vertical coordinate of $l$.
According to Lemma \ref{lemm3.2}, all the following constraints,
with respect to any horizontal l-edge $l$ in $\mathscr{T}^c$,
$$\int_{l} g(s, \bar{y}-)ds = \int_{l} g(s,
\bar{y}+)ds$$ are equivalent with the integrations of $g$ along all
the span formed by two neighboring horizontal l-edges in
$\mathscr{T}^c$ are zero. A similar equivalence can be made for
vertical l-edges. Here we should noted that, for any boundary
l-edge, there is only one constraint stated in Lemma \ref{lemm3.1}.
Hence the necessary and sufficient conditions that $\mathcal{I}(g)
\in \overline{\mathbf{S}}(1,1,0,0,\mathscr{T})$ for $g \in
\overline{\mathbf{S}}(0,0,-1,-1,\mathscr{T})$ are $E+2$ constraints
in this fashion. On the other hand, if $\mathscr{T}^c =
(x_0,x_1,\ldots, x_m) \times (y_0, y_1, \ldots, y_n)$, it follows
that
$$ \int_{-\infty}^{+\infty}\int_{-\infty}^{+\infty} g(s,t)dsdt =
\sum_{i=0}^{m-1}(x_{i+1}-x_i) C_i = \sum_{j=0}^{n-1}(y_{j+1}-y_j)
D_j,$$ where
\begin{align*}
C_i & = \int_{y_0}^{y_n} g(x,y)dy, \quad x \in (x_i, x_{i+1}), \\
D_j & = \int_{x_0}^{x_m} g(x,y)dx, \quad y \in (y_j, y_{j+1}).
\end{align*}
Hence these $E+2$ constraints is with defective rank at least one.
The latter dimension theorem \ref{thm3.4} will show that the
defective rank is exactly one.

Up to now, we have finished the description of the constraints that
$\mathcal{I}(g) \in \overline{\mathbf{S}}(1,1,0,0,\mathscr{T})$.
Before we state the dimension theorem, a topological equation of
T-meshes is proposed in the following lemma.
\begin{lemma}
\label{lemm3.3} Given a regular T-mesh $\mathscr{T}$, suppose
$\mathscr{T}$ has $F$ cells, $V^+$ crossing vertices, and $E$
interior l-edges. Then
$$F=V^+ +E+1.$$
\end{lemma}

\begin{proof}
Suppose, in $\mathscr{T}$, there are $V^T$ T-vertices, $V^{bT}$
boundary vertices (excluding four corner points). Since every cell
has four vertices, running through all the cells will meet every
crossing vertices four times, every interior T-vertices twice, every
corner points once, and every other boundary vertices twice. Hence
it follows that
$$4F=4V^+ +2V^T +2V^{bT}+4.$$
On the other hand, the end-points of every interior l-edge are
either interior T-vertices or boundary vertices (not corner points).
Therefore, we have $V^T+V^{bT}=2E$. From these two equations one
gets $F=V^++E+1$.
\end{proof}

\subsection{Dimension theorem}

\begin{theorem}
\label{thm3.4} Given a regular T-mesh $\mathscr{T}$ with $V^+$
crossing vertices, it follows that $$\dim
\overline{\mathbf{S}}(1,1,0,0,\mathscr{T})=V^+.$$
\end{theorem}

\begin{proof}
We first prove that $\dim
\overline{\mathbf{S}}(1,1,0,0,\mathscr{T})\leqslant V^+$. Suppose a
function $f\in\overline{\mathbf{S}}(1,1,0,0,\mathscr{T})$ reaches
zero at all the crossing vertices. We intend to prove that
$f\equiv0$.

Suppose $f\not\equiv0$. Without loss of generality, we assume $f$ is
greater than zero in some regions in $\mathscr{T}$. Then there
exists a point $p$ in $\mathscr{T}$ such that $f(p)=\delta=\max
f>0$. Let $p$ be in the cell $c$. There will only happen the
following two cases:
\begin{enumerate}
\item there exists an edge $e$ of $c$ such that $f$ is a constant
$\delta$ along $e$;
\item for any edge $e$ of $c$, $f$ is not a constant along $e$.
\end{enumerate}

Consider Case 1. Let $l_0$ denote the l-edge on which $e$ lies.
Define a set $L$ which consists of all the l-edges on which $f$ are
constants $\delta$.  $P$ consists of all the end-points of the
l-edges in $L$. Since $l_0 \in L$, both $L$ and $P$ are non-empty.
Now we category the vertices in $P$ into two types. If a vertex in
$P$ is also an interior vertex on another l-edge in $L$, then the
vertex is called a flat vertex. Otherwise it is called a non-flat
vertex. In the following we will prove that there must exist at
least one non-flat vertex in $P$. If so, select one non-flat vertex
$q$. Then $q$ must lie on an l-edge $l_1$ which is not in $L$. $q$
is an interior point on $l_1$. Since $f(q) = \delta$ and $f|_{l_1}$
is not a constant, there exists a point $r$ on $l_1$ such that
$f(r)> \delta$, which contracts with the fact that $\delta$ is the
maximum of $f$ over $\mathscr{T}$.

In fact, the non-flat vertex $q$ can be selected to the vertex in
$P$ with the minimal $y$ coordinate. If there exist more than one
vertex with the minimal $y$ coordinate, one with the minimal $x$
coordinate among them is selected. Such the selection ensures that
$q$ is a non-flat vertex. If not, $q$ lies on two l-edges $l_2$ and
$l_3$ in $L$, where $l_2$ is horizontal and $l_3$ vertical. If $q$
is a horizontal T-vertex, then the bottom end-point of $l_3$ has a
smaller $y$ coordinate than $q$. If $q$ is a vertical T-vertex, then
the left end-point of $l_2$ has the same $y$ coordinate as $q$, but
with smaller $x$ coordinate than $q$. Both the cases contradict with
the selection of $q$. Hence $q$ is a non-flat point. Then we can
prove $f \equiv 0$ for Case 1.

Now we consider Case 2. Since $f|_c$ is a bilinear function and $f$
is not a constant along any edges of $c$, $f$ reaches its maximum
only at one of its corners. Hence $p$ is a corner of $c$, which is a
T-vertex, say a horizontal T-vertex. The vertical l-edge through $p$
is assumed to be $l_4$, which takes $p$ as its interior point. Since
$f(p) = \delta$ and $f|_{l_4}$ is not a constant, there exists a
point $s$ on $l_4$ such that $f(s)> \delta$, which contracts with
the fact that $\delta$ is the maximum of $f$ over $\mathscr{T}$.

Summarizing the consideration for both the cases, it follows that
$f\equiv0$. Then all the crossing vertices in $\mathscr{T}$ form a
determining set of the spline space
$\overline{\mathbf{S}}(1,1,0,0,\mathscr{T})$. According to the
theory of the determining sets in spline functions
\cite{alfeld,alfeld1}, one gets
\begin{align}
\dim \overline{\mathbf{S}}(1,1,0,0,\mathscr{T})\leqslant V^+.
\label{S11}
\end{align}

On the other hand, we will prove that $\dim
\overline{\mathbf{S}}(1,1,0,0,\mathscr{T})\geqslant V^+$ in the
following. To do so, we consider the operator of mixed partial
derivative defined in Equation \eqref{mpd} as $m=n=1$:
$$\mathcal{D}: \overline{\mathbf{S}}(1,1,0,0,\mathscr{T}) \rightarrow
\overline{\mathbf{S}}(0,0,-1,-1,\mathscr{T}).$$ Here $\mathcal{D}$
is injective. The spline space
$\overline{\mathbf{S}}(0,0,-1,-1,\mathscr{T})$ is with dimension
$F$, the number of the cells in $\mathscr{T}$. According to the
analysis in the former section, in order to ensure $\mathcal{I}(g)
\in \overline{\mathbf{S}}(1,1,0,0,\mathscr{T})$, one needs to
satisfy $E+2$ constraints, which have defective rank at least one.
Then according to Lemma \ref{lemm3.3},
\begin{equation} \label{eqn3.61}
 \dim
\overline{\mathbf{S}}(1,1,0,0,\mathscr{T}) \geqslant F-(E+2)+1 =
V^+.
\end{equation}
Combining Equations \eqref{S11} and \eqref{eqn3.61}, it follows that
the dimension theorem is proved, i.e., $$\dim
\overline{\mathbf{S}}(1,1,0,0,\mathscr{T})=V^+.$$
\end{proof}

\subsection{Basis functions}

\label{sec3.4} We can construct a set of basis functions
$\{b_i(x,y)\}_{i=1}^{V^+}$ for the spline space
$\overline{\mathbf{S}}(1,1,0,0,\mathscr{T})$. The basis functions
should have the following properties:
\begin{enumerate}
\item \textbf{Compact Support}: For any $i$, $b_i(x,y)$ has a support
as small as possible;

\item \textbf{Nonnegativity}: For any $i$, $b_i(x,y) \geqslant 0$;

\item \textbf{Partition of Unity}: If $\mathscr{T}$ is an extension of some T-mesh
$\mathscr{T}_0$ with respect to the spline space
$\overline{\mathbf{S}}(1,1,0,0,\mathscr{T})$, and the region
occupied by $\mathscr{T}_0$ is $\Omega$, then
$$ \sum_{i=1}^{V^+} b_i(x,y) =1, \quad (x,y) \in \Omega. $$
\end{enumerate}

The basis functions with those properties can be constructed as
follows: Suppose the crossing vertices in $\mathscr{T}$ are $v_i$
with coordinate $(x_i, y_i)$, $i=1, \ldots, V^+$. Then we require
the function $b_i(x,y)$ satisfy $ b_i(x_j, y_j) = \delta_{ij}$.
According to the dimension theorem \ref{thm3.4} and the first part
in its proof, $b_i(x,y)$ is determined uniquely. All the functions
$b_i(x,y)$ form a set of basis functions of the spline space, which
have the former three properties. In fact, it is easy to show that
Properties 1 and 2 are satisfied. Now we prove Property 3.

\begin{theorem}
Given a regular T-mesh $\mathscr{T}$, which occupies a rectangle
$\Omega$, its extension associated with
$\mathbf{S}(1,1,0,0,\mathscr{T})$ is $\mathscr{T}^{\varepsilon}$.
The crossing vertices in $\mathscr{T}^{\varepsilon}$ are $v_i$ with
coordinate $(x_i, y_i)$, $i=1, \ldots, V^+$. The function set
$\{b_i(x,y)\}_{i=1}^{V^+} \subset
\overline{\mathbf{S}}(1,1,0,0,\mathscr{T}^\varepsilon)$ satisfy
$b_i(x_j,y_j) = \delta_{ij}$. Then
\begin{equation}
\sum_{i=1}^{V^+} b_i(x,y) =1, \quad (x,y) \in \Omega. \label{eqn3.7}
\end{equation}
\end{theorem}

\begin{proof}
Let $$f(x,y) = \sum_{i=1}^{V^+} b_i(x,y),$$ and $\ell$ denote the
boundary of $\mathscr{T}$. In order to show Equation \eqref{eqn3.7}
holds, we first prove $f|_\ell \equiv 1$. Because the vertices on
$\ell$ are crossing vertices in $\mathscr{T}^\varepsilon$, it
follows that $f$ reaches 1 on these vertices. Since $f|_\ell$ is a
piecewise linear function with knots being these vertices, it
follows that $f|_\ell \equiv 1$. Then, in the following, we prove
that, for any $(x,y) \in \Omega$, $f(x,y)=1$. Consider the function
$$g(x,y) = \begin{cases}
f(x,y)-1 & (x,y) \in \Omega \\
0 & \mbox{otherwise}
\end{cases}
$$
Since $f|_\ell \equiv 1$, one has $g \in
\overline{\mathbf{S}}(1,1,0,0,\mathscr{T})$. But $g$ is zero at all
the crossing vertices of $\mathscr{T}$, it follows that, according
to the proof of the dimension theorem \ref{thm3.4}, we have $g
\equiv 0$, i.e., $f(x,y) =1$, $(x,y) \in \Omega$.
\end{proof}

\noindent \textbf{Remarks}: There are some interesting problems open
here. \begin{enumerate} \item How to directly specify the former
basis functions in every cell of a general T-mesh? \item How to
evaluate the function or the surface which is represented in the
linear combination of the former basis functions? \item What is the
``knot'' insertion algorithm in this spline space?
\end{enumerate}
Though the space
is just bilinear, the solutions to these problems will possibly hint
us how to do in higher degree spline spaces over T-meshes.

\section{A Lower Bound of the Dimension of $\overline{\mathbf{S}}(2,2,1,1,\mathscr{T})$}
\label{sec4}

We can apply a similar method proposed in the proof of the dimension
theorem of $\overline{\mathbf{S}}(1,1,0,0,\mathscr{T})$ in the
former section to the dimension analysis of
$\overline{\mathbf{S}}(2,2,1,1,\mathscr{T})$. After that, we can
obtain a lower bound of the dimension, i.e.,
\begin{equation}
\dim \overline{\mathbf{S}}(2,2,1,1,\mathscr{T})\geqslant V^+-E+1.
\end{equation}

\subsection{Some lemmas}

Now we consider the operator of mixed partial derivative as follows:
$$\mathcal{D}:=\frac{\partial^2}{\partial x\partial y}:
\overline{\mathbf{S}}(2,2,1,1,\mathscr{T}) \rightarrow
\overline{\mathbf{S}}(1,1,0,0,\mathscr{T}).$$ Here $\mathcal{D}$ is
injective as well. The operator $\mathcal{I}(g)$ is defined in the
same way as Equation \eqref{ig}. The following lemmas discuss the
constraints ensuring $\mathcal{I}(g) \in
\mathbf{S}(2,2,1,1,\mathscr{T})$ for any $g \in
\mathbf{S}(1,1,0,0,\mathscr{T})$.

\begin{lemma}
\label{lemm4.1} Given a regular T-mesh $\mathscr{T}$, let the
coordinate of the end-points of any horizontal l-edges $l^h_{i}$ be
$(x^h_{i1},y^h_i)$ and $(x^h_{i2},y^h_i)$, $i=0,1,\ldots,m$, and the
coordinate of the end-points of any vertical l-edge $l^v_{j}$ be
$(x^v_j,y^v_{j1})$ and $(x^v_j,y^v_{j2})$, $j=0,1,\ldots,n$. For any
$g \in \overline{\mathbf{S}}(1,1,0,0,\mathscr{T})$, it follows that
\begin{equation*} \mathcal{I}(g) \in
\overline{\mathbf{S}}(2,2,1,1,\mathscr{T}) \Leftrightarrow \left\{
\begin{array}{lll}
\displaystyle  \int^{x^h_{i2}}_{x^h_{i1}}\frac{\partial}{\partial
y}g(s,y^h_i-){d}s=\int^{x^h_{i2}}_{x^h_{i1}}\frac{\partial}{\partial
y}g(s,y^h_i+){d}s, \quad i=0,1,\ldots,m, \mbox{and}\\[4mm]
  \displaystyle    \int^{y^v_{j2}}_{y^v_{j1}}\frac{\partial}{\partial
x}g(x^v_j-,t)dt=\int^{y^v_{j2}}_{y^v_{j1}}\frac{\partial}{\partial
x}g(x^v_i+,t){d}t, \quad j=0,1,\ldots,n.
\end{array}
\right.
\end{equation*}
\end{lemma}
\begin{proof} We first prove the necessity $\Longrightarrow$. Let $f(x,y)=\mathcal{I}(g)(x,y)$.
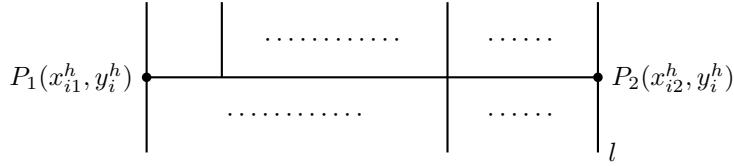
\begin{figure}[!htb]
\begin{center}
\setlength{\unitlength}{1cm}
\begin{pspicture}(1,1)(8,3)
\psline(1,2)(7,2)\psline(1,1)(1,3) \psline(7,1)(7,3)
\psline(2,2)(2,3)\rput(3.5,2.5){$\cdots\cdots\cdots\cdots$}\rput(6,2.5){$\cdots\cdots$}
\psline(5,1)(5,3)\rput(3,1.5){$\cdots\cdots\cdots\cdots$}\rput(6,1.5){$\cdots\cdots$}
\psdot(1,2)\rput(0,2){$P_1(x^h_{i1},y^h_i)$}
\psdot(7,2)\rput(8,2){$P_2(x^h_{i2},y^h_i)$} \rput(7.2,1){$l$}
\end{pspicture}
\caption{A horizontal l-edge (2)\label{HL2}}
\end{center}
\end{figure}
Without loss of generality, we only prove that, when $f \in
\overline{\mathbf{S}}(2,2,1,1,\mathscr{T})$, the constraints
corresponding to horizontal l-edges are satisfied. As shown in
Figure \ref{HL2}, the two end-points of the horizontal l-edge are
$P_1(x^h_{i1},y^h_i)$ and $P_2(x^h_{i2},y^h_i)$. The vertical edge
on which $P_2$ lies is $l$. Because
$f\in\overline{\mathbf{S}}(2,2,1,1,\mathscr{T})$, it follows that
$f(x_{i2}^h,y)$ is a quadratic polynomial with respect to the
variable $y$ in a neighborhood of $P_2$ along $l$. Hence,
\begin{align*}
f(x^h_{i2},y^h_i-)&=f(x^h_{i2},y^h_i+),\\
\frac{\partial}{\partial
y}f(x^h_{i2},y^h_i-)&=\frac{\partial}{\partial
y}f(x^h_{i2},y^h_i+),\\
\frac{\partial^2}{\partial
y^2}f(x^h_{i2},y^h_i-)&=\frac{\partial^2}{\partial
y^2}f(x^h_{i2},y^h_i+).
\end{align*}
According to the definition of $f$ and the continuous of $g$, the
first two equations hold trivially. Substituting the definition of
$f$ into the last equation, one has
\begin{equation}\label{int4}
\int_{-\infty}^{x^h_{i2}}\frac{\partial}{\partial y}g(s,y^h_i-)ds
 =\int_{-\infty}^{x^h_{i2}}\frac{\partial}{\partial y}g(s,y^h_i+)ds.
\end{equation}

Extend the current l-edge to the left boundary of the T-mesh. If
there does not exist any other l-edges on the extension, then in
every cell which intersects the extension on the left side of
$(x^h_{i1},y^h_i)$, $g$ is a single bilinear function. Hence it
follows that
\begin{equation}\label{int5}
\int_{-\infty}^{x^h_{i1}}\frac{\partial}{\partial y}g(s,y^h_i-)ds
 =\int_{-\infty}^{x^h_{i1}}\frac{\partial}{\partial y}g(s,y^h_i+)ds.
\end{equation}
Subtracting Equation \eqref{int5} from Equation \eqref{int4}, one
gets
$$
\int_{x^h_{i1}}^{x^h_{i2}}\frac{\partial}{\partial y}g(s,y^h_i-)ds
 =\int_{x^h_{i1}}^{x^h_{i2}}\frac{\partial}{\partial y}g(s,y^h_i+)ds.
$$
This proves that the equation in the lemma corresponding to the
current l-edge holds. If there exists some other l-edges on the
extension, then we consider these l-edges one by one from left to
right. We can prove that the former equations hold for these
l-edges. Hence for the current l-edge, according to Equation
\eqref{int4}, the same equation holds as well.

Now we prove the sufficiency $\Longleftarrow$. We will show that,
for any $g \in \overline{\mathbf{S}}(1,1,0,0,\mathscr{T})$, if the
following equations hold:
\begin{align}
\label{eqn4.2} &\int^{x^h_{i2}}_{x^h_{i1}}\frac{\partial}{\partial
y}g(s,y^h_i-){d}s=\int^{x^h_{i2}}_{x^h_{i1}}\frac{\partial}{\partial
y}g(s,y^h_i+){d}s, \quad i=0,1,\ldots,m,\\
\label{eqn4.3} &\int^{y^v_{j2}}_{y^v_{j1}}\frac{\partial}{\partial
x}g(x^v_j-,t)dt=\int^{y^v_{j2}}_{y^v_{j1}}\frac{\partial}{\partial
x}g(x^v_i+,t){d}t, \quad j=0,1,\ldots,n,
\end{align}
then $f$ is a single biquadratic polynomial in every cell of
$\mathscr{T}$.

Otherwise, suppose in some cell $c_i$, $f$ is piecewise with
horizontal discontinuous-lines $y=y_{i_1},\ldots,y=y_{i_{n_i}}$ and
vertical discontinuous-lines $x=x_{i_1},\ldots,x=x_{i_{m_i}}$ of the
partial derivatives of order two. Let $\mathscr{C}_y=\bigcup_i
\{y_{i_1},\ldots,y_{i_{n_i}}\}$, $\mathscr{C}_x=\bigcup_i
\{x_{i_1},\ldots,x_{i_{m_i}}\}$. Then according to the assumption,
$\mathscr{C}_x \bigcup \mathscr{C}_y$ is non-empty. Without loss of
generality, we assume $\mathscr{C}_y$ is non-empty. Let
$\bar{y}=\min\limits_{y\in\mathscr{C}_y}y$. Suppose $c_i$ is the
leftmost cell which takes $y=\bar{y}$ as its inner
discontinuous-line of the partial derivatives of order two with
respect to $x$. Denote the $x$ coordinate of the left boundary edge
of $c_i$ is $x_0$, and $(x_1,\bar{y})$ is a point on the
discontinuous-line, where $x_1>x_0$. Then it follows that
\begin{equation*}
\int_{-\infty}^{x_0}\frac{\partial}{\partial y}g(s,\bar{y}-)ds
                        =\int_{-\infty}^{x_0}\frac{\partial}{\partial
                        y}g(s,\bar{y}+)ds.
\end{equation*}
Since $(x_1,\bar{y})$ is a discontinuous point, it follows that
\begin{equation*}
\int_{-\infty}^{x_1}\frac{\partial}{\partial y}g(s,\bar{y}-)ds
\neq\int_{-\infty}^{x_1}\frac{\partial}{\partial y}g(s,\bar{y}+)ds.
\end{equation*}
Hence
$$
\int_{x_0}^{x_1}\frac{\partial}{\partial y}g(s,\bar{y}-)ds
 \neq\int_{x_0}^{x_1}\frac{\partial}{\partial y}g(s,\bar{y}+)ds.
$$
This contradicts with the fact that $\frac{\partial}{\partial
y}g(x,y)$ is continuous inside the cell $c_i$. Therefore we have
proved that $f$ is a single biquadratic polynomial in every cell of
$\mathscr{T}$.

In order to prove $f \in
\overline{\mathbf{S}}(2,2,1,1,\mathscr{T})$, one needs further to
verify that $f$ satisfies HBC. In fact, with the same approach, one
can prove that $f$ is a single biquadratic polynomial outside
$\mathscr{T}$. On the other hand, $f(x,y)$ is zero when $x \leqslant
x_0, y \leqslant y_0$, where $(x_0,y_0)$ is the coordinate of the
left-bottom corner of the T-mesh $\mathscr{T}$. Hence $f$ is
everywhere zero outside $\mathscr{T}$, which ensures that $f$
satisfies HBC. Therefore $f \in
\overline{\mathbf{S}}(2,2,1,1,\mathscr{T})$.
\end{proof}

According to the former lemma and the fact $\dim
\overline{\mathbf{S}}(1,1,0,0,\mathscr{T}) = V^+$, in order to
ensure $f\in\overline{\mathbf{S}}(2,2,1,1,\mathscr{T})$, there are
$E+4$ constraints with $V^+$ under-determining coefficients, where
$E$ is the number of interior l-edges. But these constraints are not
linear independent, as stated in the following lemma.

\begin{lemma}
\label{lemm4.2} Given a regular T-mesh $\mathscr{T}$, whose occupied
rectangle is $(x_l,x_r) \times (y_b, y_t)$. Assume the different $y$
coordinates from all the horizontal l-edges are $y_0 < y_1 < \ldots
< y_n$. For $g \in \overline{\mathbf{S}}(1,1,0,0,\mathscr{T})$, it
follows that
\begin{equation}
\label{eqn4.4} \int^{x_r}_{x_l}\frac{\partial}{\partial
y}g(s,y_i-)ds=\int^{x_r}_{x_l}\frac{\partial}{\partial
y}g(s,y_i+)ds, \quad i=0,\ldots,n\end{equation} is equivalent with
\begin{equation}
\label{eqn4.5} \int^{x_r}_{x_l} g(s,y_i)ds=0,\quad i=1,\ldots,n-1.
\end{equation}
The similar conclusion can be made for vertical l-edges.
\end{lemma}

\begin{proof}
We first prove the necessity. Getting started from the bottom
boundary l-edge $l_0$ we have
$$\int^{x_r}_{x_l}\frac{\partial}{\partial y}g(s,y_0-)ds=0, \quad
\int^{x_r}_{x_l}g(s,y_0)ds=0.$$ Then according to Equation
\eqref{eqn4.4} as $i=0$, one gets
$$\int^{x_r}_{x_l}\frac{\partial}{\partial
y}g(s,y_0+)ds=0. $$ On the other hand, according to the piecewise
bilinear definition of $g$, one has
\begin{align}
\nonumber \int^{x_r}_{x_l}\frac{\partial}{\partial y}g(s,y_0+)ds
&=\int^{x_r}_{x_l}\frac{\partial}{\partial
y}\bigg(\frac{y_1-y}{y_1-y_0}g(s,y_0)+\frac{y-y_0}{y_1-y_0}g(s,y_1)\bigg)ds \\
\label{eqn4.6}
&=\frac{1}{y_1-y_0}\bigg(\int^{x_r}_{x_l}g(s,y_1)ds-\int^{x_r}_{x_l}g(s,y_0)ds\bigg)
\\
&= \frac{1}{y_1-y_0}\int^{x_r}_{x_l}g(s,y_1)ds. \nonumber
\end{align}
Therefore $$\int^{x_r}_{x_l}g(s,y_1)ds=0.$$ Recursively, one can
prove that Equation \eqref{eqn4.5} holds for every horizontal
l-edge.

In order to prove the sufficiency, we take the similar deduction in
Equation \eqref{eqn4.6}. One obtains
\begin{align*} \int^{x_r}_{x_l}\frac{\partial}{\partial
y}g(s,y_i-)ds & =
\frac{1}{y_i-y_{i-1}}\bigg(\int^{x_r}_{x_l}g(s,y_i)ds-\int^{x_r}_{x_l}g(s,y_{i-1})ds\bigg),
\\
\int^{x_r}_{x_l}\frac{\partial}{\partial y}g(s,y_i+)ds & =
\frac{1}{y_{i+1}-y_i}\bigg(\int^{x_r}_{x_l}g(s,y_{i+1})ds-\int^{x_r}_{x_l}g(s,y_i)ds\bigg).
\end{align*}
Therefore Equation \eqref{eqn4.5} ensures Equation \eqref{eqn4.4}
holds.
\end{proof}

\noindent \textbf{Remark}: In fact, from the proof of the lemma, we
can also conclude that, for horizontal l-edges,
$$
\int^{x_r}_{x_l}\frac{\partial}{\partial
y}g(s,y_i-)ds=\int^{x_r}_{x_l}\frac{\partial}{\partial
y}g(s,y_i+)ds, \quad i=0,\ldots,n $$ is equivalent with
$$
\int^{x_r}_{x_l}\frac{\partial}{\partial y}g(s,y_i-)ds = 0,\quad
i=1,\ldots,n-1.
$$

Lemma \ref{lemm4.2} states that there are at least two redundant
constraints among those corresponding to horizontal l-edges and
vertical l-edges, respectively. The following lemma tells us that we
can take the constraint along four boundary l-edges as redundant
ones.

\begin{lemma}
\label{lemm4.3} Given a regular T-mesh $\mathscr{T}$, the rectangle
occupied by it is $(x_l,x_r) \times (y_b, y_t)$, where the different
$y$ coordinates from horizontal l-edges are $y_0 < y_1 < \ldots <
y_n$. Let $g \in \overline{\mathbf{S}}(1,1,0,0,\mathscr{T})$. Then
Equation \eqref{eqn4.4} is equivalent with
\begin{equation}
\label{eqn4.7} \int^{x_r}_{x_l}\frac{\partial}{\partial
y}g(s,y_i-)ds=\int^{x_r}_{x_l}\frac{\partial}{\partial
y}g(s,y_i+)ds, \quad i=1,\ldots,n-1.
\end{equation}
The similar conclusion can be made for vertical l-edges.
\end{lemma}

\begin{proof}
The necessity is obvious. In order to prove the sufficiency, we show
that Equation \eqref{eqn4.7} implies that Equation \eqref{eqn4.5}
holds. In fact, suppose
$$I_i = \int_{x_l}^{x_r} g(s,y_{i})ds, i=0,1,\ldots,n. $$
Then $I_0 = I_n= 0$. Our object is to prove that, for any
$i=1,2,\ldots,n-1$, $I_{i}=0$. Taking the similar deduction in
Equation \eqref{eqn4.6}, one gets
\begin{align*} \int_{x_l}^{x_r} g(s,y_{i})ds &
=\frac{y_{i+1}-y_{i}}{y_{i+1}-y_{i-1}} \int_{x_l}^{x_r}
g(s,y_{i-1})+\frac{y_{i}-y_{i-1}}{y_{i+1}-y_{i-1}} \int_{x_l}^{x_r}
g(s,y_{i+1})ds,
\end{align*}
i.e., $$I_i =\frac{y_{i+1}-y_{i}}{y_{i+1}-y_{i-1}}
I_{i-1}+\frac{y_{i}-y_{i-1}}{y_{i+1}-y_{i-1}} I_{i+1} $$ holds for
any $i=1,2,\ldots,n-1$. Hence the point set $\{(y_i,I_i)\}_{i=0}^n$
is collinear, which means that there exists a linear function $f(y)$
such that $f(y_i) = I_i$, $i=0, \ldots, n$. Since $y_0 = y_n = 0$,
it follows that $f(y) \equiv 0$, which means that $I_i=0$,
$i=1,2,\ldots,n$. Then the sufficiency is proved.
\end{proof}

\subsection{A lower bound of dimensions}
\label{sec4.2}

For a given regular T-mesh $\mathscr{T}$, which has $E$ interior
l-edges, we assume its associated tensor-product mesh
$\mathscr{T}^c$ has $E'$ interior l-edges, where $E' \leqslant E$.
Suppose the horizontal l-edges $l_1, \ldots, l_k$ in $\mathscr{T}$
lie on the horizontal l-edge $l$ in $\mathscr{T}^c$, where the
vertical coordinate of $l$ is $\bar{y}$. Then one gets
$$ \int_{l_i} \frac{\partial}{\partial y}g(s, \bar{y}-)ds = \int_{l_i}  \frac{\partial}{\partial y}
g(s, \bar{y}+)ds, \quad i=1, \ldots, k$$ is equivalent with
$$ \int_{l_i}  \frac{\partial}{\partial y} g(s, \bar{y}-)ds = \int_{l_i}  \frac{\partial}{\partial y}
g(s, \bar{y}+)ds, \quad i=1,\ldots, k-1, \quad
\int_{l}\frac{\partial}{\partial y}  g(s, \bar{y}-)ds = \int_{l}
\frac{\partial}{\partial y} g(s, \bar{y}+)ds. $$ According to Lemma
\ref{lemm4.2}, the constraints, along all the interior horizontal
l-edges and top/bottom boundary l-edges in $\mathscr{T}^c$, $$
\int_{l}\frac{\partial}{\partial y}  g(s, \bar{y}-)ds = \int_{l}
\frac{\partial}{\partial y} g(s, \bar{y}+)ds
$$
is equivalent with the integrations of $g$ along all the interior
horizontal l-edges in $\mathscr{T}^c$ are zero:
$$\int_l g(s,\bar{y}) ds = 0.$$
The similar conclusion can be made for vertical l-edges. Therefore
the number of the sufficient and necessary constraints that ensure
$\mathcal{I}(g) \in \overline{\mathbf{S}}(2,2,1,1,\mathscr{T})$ for
$g \in \overline{\mathbf{S}}(1,1,0,0,\mathscr{T})$ is just $E$.

Moreover, in the tensor-product mesh $\mathscr{T}^c$, if the
horizontal knots are $x_0<x_1<\cdots<x_m$, and the vertical knots
are $y_0< y_1 < \cdots < y_n$, then we have
\begin{equation} \label{lll}
\int_{-\infty}^{+\infty}\int_{-\infty}^{+\infty} g(s,t)dsdt =
\sum_{i=1}^{m-1}(x_{i+1}-x_{i-1}) E_i =
\sum_{j=1}^{n-1}(y_{j+1}-y_{j-1}) F_j, \end{equation}where
$$
E_i = \int_{y_0}^{y_n} g(x_i,y)dy, \quad F_j = \int_{x_0}^{x_m}
g(x,y_j)dx.
$$
Hence, among the former $E$ constraints, the element number in any
maximal linearly independent subset is at most $E-1$ when these
constraints are not empty (i.e., $V^+>0$).

Based on these analysis, a lower bound of dimensions of spline
spaces $\mathbf{S}(2,2,1,1,\mathscr{T})$ can be proposed as follows:
\begin{theorem}
\label{Dim2211} Given a regular T-mesh $\mathscr{T}$ with $V^+ > 0$
crossing vertices and $E$ interior l-edges, it follows that
$$
\dim \overline{\mathbf{S}}(2,2,1,1,\mathscr{T}) \geqslant V^+-E+1.
$$
\end{theorem}

\begin{proof}
Since $V^+>0$, the constraints are not empty, and $E>1$. According
the former analysis and Theorem \ref{thm3.4}, the dimension of
$\overline{\mathbf{S}}(1,1,0,0,\mathscr{T})$ is $V^+$. For any $g
\in \overline{\mathbf{S}}(1,1,0,0,\mathscr{T})$, the constraints
ensuring $\mathcal{I}(g) \in
\overline{\mathbf{S}}(2,2,1,1,\mathscr{T})$ have maximal linearly
dependent subsets with element number at most $E-1$. Here both
$\mathcal{D}$ and $\mathcal{I}$ are linearly injective. Therefore
$\dim \overline{\mathbf{S}}(2,2,1,1,\mathscr{T}) \geqslant V^+-E+1$.
\end{proof}

The lower bound in the theorem can be reached for spline spaces over
some T-meshes. For example, consider the T-mesh $\mathscr{T}$ as
shown in Figure \ref{TExpand}. The B-net method in \cite{Deng06} can
show that $\dim \overline{\mathbf{S}}(2,2,1,1,\mathscr{T}) = 1$. On
the other hand, in $\mathscr{T}$, $V^+=6$, $E=6$. Then $V^+-E+1=1$,
which means that the lower bound is reached in this T-mesh.

Furthermore, it is easy to verify that, for tensor-product meshes,
the former lower bound is exactly the same as dimensions of
biquadratic spline spaces over the meshes, if $V^+-E +1 \geqslant
0$. In the next section, we will prove that, for hierarchical
T-meshes, the lower bound can be reached in some cases as well.


\section{Dimensions of Spline Spaces $\overline{\mathbf{S}}(2,2,1,1,\mathscr{T})$ over Hierarchical T-meshes}
\label{ht}

In this section, a careful analysis on the constraints in Section
\ref{sec4} will help us to build a dimension formula of the spline
space $\overline{\mathbf{S}}(2,2,1,1,\mathscr{T})$ over a
hierarchical T-mesh. Here the key procedure consists of the
following components:
\begin{enumerate}
\item A general hierarchical T-mesh is divided into some
so-called crossing-vertex-connected branches (Definition \ref{cvc}).
Then the spline space over the hierarchical T-mesh can be divided
into the direct sum of some subspaces, each of which is defined over
a crossing-vertex-connected hierarchical branch, which is a T-mesh
as well. See Subsection \ref{gdt}.

\item In a crossing-vertex-connected hierarchical T-mesh, the
constraint set is proved to with defective rank exact one by the
following processing:

\begin{enumerate}
\item The constraints are converted into a new form to reflect the
level structure of the hierarchical T-mesh. See Subsection
\ref{sec5.2}.

\item A new set of basis functions of $\overline{\mathbf{S}}(1,1,0,0,\mathscr{T})$
is defined according the structure of the T-mesh, such that the
occurrence of the basis function coefficients in the constraints is
regularized. See Subsection \ref{hbf} and Proposition \ref{prop5.6}.

\item All the l-edges and the corresponding constraints are ordered
according to the structure of the T-mesh. For each of the l-edges or
the constraints, a characteristic vertex is introduced. Write all
the constraints in a vector form $(C_0,C_1, \ldots, C_T)^T$ in the
increasing order. The constraint $C_0$ can be removed since we have
known all the constraints are with defective reank at least one.
Hence we need to shown that $(C_1, \ldots, C_T)^T$ is with full
rank. See Subsection \ref{ordering}.

\item Assume the characteristic vertex of $C_i$ is $V_i$, $i=1,\ldots,
T$. Arrange the coefficients into a vector $(\beta_1,\ldots,
\beta_T, \beta_{T+1}, \ldots, \beta_M)^T$ such that $\beta_i$ is the
coefficient of the basis function, corresponding to $V_i$,
$i=1,2,\ldots,T$, in $\overline{\mathbf{S}}(1,1,0,0,\mathscr{T})$.

\item
Then there exists a matrix $\mathbf{M} = (m_{ij})_{T\times M}$ such
that $(C_1, \ldots, C_T)^T = \mathbf{M} (\beta_1,\ldots,
\beta_M)^T$. It can be shown that $m_{ij}=0$, $i>j$, and $m_{ii}
\neq 0$ for any $i$ and $j$.
\end{enumerate}
Hence the set of all the constraints is with defective rank one.
\end{enumerate}

\subsection{Hierarchical T-meshes}
\begin{figure}[!ht]
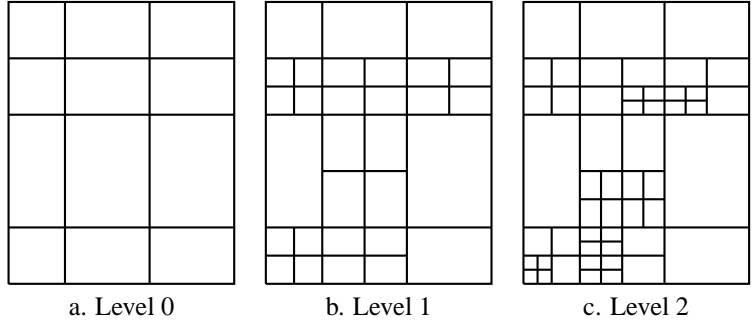

\begin{center}
\psset{unit=0.75cm,linewidth=0.8pt}
\begin{tabular}{ccc}
\pspicture(0,0)(4,5) \psline(0,0)(4,0)(4,5)(0,5)(0,0)
\psline(0,1)(4,1)\psline(0,3)(4,3)\psline(0,4)(4,4)
\psline(1,0)(1,5)\psline(2.5,0)(2.5,5)
\endpspicture &
\pspicture(0,0)(4,5)
\psline(0,0)(4,0)(4,5)(0,5)(0,0)\psline(0,1)(4,1)\psline(0,3)(4,3)\psline(0,4)(4,4)
\psline(1,0)(1,5)\psline(2.5,0)(2.5,5)\psline(1.75,0)(1.75,4)
\psline(0,0.5)(2.5,0.5)\psline(0.5,0)(0.5,1)\psline(2.5,2)(1,2)\psline(0,3.5)(4,3.5)
\psline(0.5,3)(0.5,4)\psline(3.25,3)(3.25,4)
\endpspicture &\pspicture(0,0)(4,5) \psline(0,0)(4,0)(4,5)(0,5)(0,0)
\psline(0,1)(4,1)\psline(0,3)(4,3)\psline(0,4)(4,4)
\psline(1,0)(1,5)\psline(2.5,0)(2.5,5)\psline(1.75,0)(1.75,4)
\psline(0,0.5)(2.5,0.5)\psline(0.5,0)(0.5,1)\psline(2.5,2)(1,2)\psline(0,3.5)(4,3.5)
\psline(0.5,3)(0.5,4)\psline(3.25,3)(3.25,4)\psline(0,0.25)(0.5,0.25)
\psline(0.25,0)(0.25,0.5)\psline(1.375,0)(1.375,2)\psline(1,0.25)(1.75,0.25)
\psline(1,0.75)(1.75,0.75)\psline(1,1.5)(2.5,1.5)\psline(2.125,1)(2.125,2)
\psline(2.125,3)(2.125,3.5)\psline(1.75,3.25)(3.25,3.25)\psline(2.875,3)(2.875,3.5)
\endpspicture \\
a. Level 0  & b. Level 1  & c. Level 2
\end{tabular}
\caption{A hierarchical T-mesh \label{HTmesh}}
\end{center}
\end{figure}
A hierarchical T-mesh \cite{Deng07} is a special type of T-mesh
which has a natural level structure. It is defined in a recursive
fashion. One generally starts from a tensor-product mesh (level
$0$). From level $k$ to level $k+1$, one subdivide a cell at level
$k$ into four subcells which are cells at level $k+1$. For
simplicity, we subdivide each cell by connecting the middle points
of the opposite edges with two straight lines. Figure \ref{HTmesh}
illustrates the process of generating a hierarchical T-mesh. For a
hierarchical T-mesh $\mathscr{T}$, in order to emphasis its level
structure in some cases, we denote the T-mesh of level $k$ to be
$\mathscr{T}^k$.

For a given hierarchical T-mesh $\mathscr{T}$, we can extend it to
obtain an extended T-mesh $\mathscr{T}^\varepsilon$ associated with
the spline space $\mathbf{S}(2,2,1,1,\mathscr{T})$. In the following
text, hierarchical T-meshes refer to both the classical hierarchical
T-meshes and their extension.

Over a hierarchical T-mesh $\mathscr{T}$, the dimension of
$\overline{\mathbf{S}}(2,2,1,1,\mathscr{T})$ may be greater than the
lower bound in Theorem \ref{Dim2211}. For example, consider the
hierarchical T-mesh as shown in Figure \ref{exam1}, where the mesh
of level $0$ is a tensor-product mesh with size $3 \times 3$, and in
the mesh of level 1, there exists only one cell that is subdivided.
In this mesh, $V^+=5$, $E=6$. The lower bound is $V^+-E+1=0$. But
the dimension of the biquadratic spline space over the mesh is at
least one obviously.
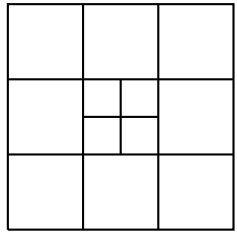
\begin{figure}[!htp]
\begin{center}
\begin{pspicture}(0,0)(3,3)
\psline(0,0)(3,0)(3,3)(0,3)(0,0)\psline(1,0)(1,3)\psline(2,0)(2,3)\psline(0,1)(3,1)\psline(0,2)(3,2)
\psline(1,1.5)(2,1.5)\psline(1.5,1)(1.5,2)
\end{pspicture}
\caption{\label{exam1}A hierarchical T-mesh where the dimension is
greater than the lower bound}
\end{center}
\end{figure}


\subsubsection{Crossing-vertex connected}

In order to ensure the dimensions of spline spaces over hierarchical
T-meshes reach the former lower bound, we need to focus on a special
type of hierarchical T-meshes.

\begin{definition}
\label{cvc} For a regular T-mesh, if, between any two different
crossing vertices, there exists a continuous poly-line, which
consists of edges in the mesh, such that every joint between two
neighboring horizontal and vertical edges on the poly-line is a
crossing vertex in the mesh, then the T-mesh is called
\textbf{crossing-vertex connected}. Such poly-lines are called
\textbf{paths} between the crossing-vertices.
\end{definition}

\begin{figure}[!ht]
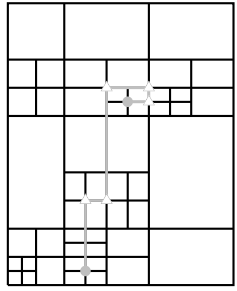

\begin{center}
\psset{unit=0.75cm,linewidth=0.8pt} \pspicture(0,0)(4,5)
\psline(0,0)(4,0)(4,5)(0,5)(0,0)
\psline(0,1)(4,1)\psline(0,3)(4,3)\psline(0,4)(4,4)
\psline(1,0)(1,5)\psline(2.5,0)(2.5,5)\psline(1.75,0)(1.75,4)
\psline(0,0.5)(2.5,0.5)\psline(0.5,0)(0.5,1)\psline(2.5,2)(1,2)\psline(0,3.5)(4,3.5)
\psline(0.5,3)(0.5,4)\psline(3.25,3)(3.25,4)\psline(0,0.25)(0.5,0.25)
\psline(0.25,0)(0.25,0.5)\psline(1.375,0)(1.375,2)\psline(1,0.25)(1.75,0.25)
\psline(1,0.75)(1.75,0.75)\psline(1,1.5)(2.5,1.5)\psline(2.125,1)(2.125,2)
\psline(2.125,3)(2.125,3.5)\psline(1.75,3.25)(3.25,3.25)\psline(2.875,3)(2.875,3.5)
\psset{linewidth=1pt,linecolor=lightgray}\psline(1.375,0.25)(1.375,1.5)(1.75,1.5)%
(1.75,3.5)(2.5,3.5)(2.5,3.25)(2.125,3.25)
\psdots(1.375,0.25)(2.125,3.25) \psdots[dotstyle=triangle](1.375,1.5)(1.75,1.5)%
(1.75,3.5)(2.5,3.5)(2.5,3.25)
\endpspicture
\caption{A hierarchical T-mesh with a path between two connected
crossing vertices\label{HTmesh1}}
\end{center}
\end{figure}

For example, two crossing vertices, labeled with light-gray dots,
are selected in the T-mesh as shown in Figure \ref{HTmesh1}. A path
between these two vertices is illustrated in light-gray as well,
where the joints are shown with light-gray triangles.

\begin{definition}
From level $k$ to level $k+1$ as forming a hierarchical T-mesh, if
there exists a cell of level $k$ to be subdivided, but all its
horizontal and vertical neighboring cells of level $k$ remain
unchanged, then the cell is called an \textbf{isolated subdivided
cell}.
\end{definition}
For a hierarchical T-mesh, it is crossing-vertex connected if and
only if, in any level of forming the hierarchical T-mesh, there do
not exist any isolated subdivided cells.

In Subsections \ref{sec5.2}--\ref{sec5.5}, we will prove that the
lower bound of dimension proposed in Theorem \ref{Dim2211} is
exactly the dimension of
$\overline{\mathbf{S}}(2,2,1,1,\mathscr{T})$ over a crossing-vertex
connected hierarchical T-mesh. Then in Subsection \ref{gdt}, a
dimension theorem can be proposed to calculate the dimension of the
spline space over a general hierarchical T-mesh by dividing it into
the union of some crossing-vertex connected hierarchical branches.

\subsubsection{Level numbers of edges, l-edges, and crossing-vertices}

\begin{figure}[!htb]
\begin{center}
\setlength{\unitlength}{1cm}
\begin{pspicture}(0,0)(6,3.5)
\psline(0.5,1)(5.5,1)\psline(1,2)(5,2) \psline(0.5,3)(5.5,3)
\psline(1,0.5)(1,3.5)\psline(2,1)(2,3)\psline(3,0.5)(3,3.5)\psline(4,1)(4,3)\psline(5,0.5)(5,3.5)
\rput[r](0.3,1){\small edge of level $p_2$}\rput[r](0.6,2){\small
edge of level $p_0$}\rput[r](0.3,3){\small edge of level $p_1$}
\rput(2,0.3){\small edge of}\rput(3.3,0.3){\small edge of}
\rput(2,0){\small level $p_0$}\rput(3.3,0){\small level $p_1$}
\psline{->}(2,0.4)(2,0.9) \psdot(2,2)\rput(1.7,1.8){$v_1$}
\psdot(3,2)\rput(2.8,1.8){$v_2$}
\end{pspicture}
\caption{Level numbers of edges, l-edges, and crossing vertices
\label{BEdgeL}}
\end{center}
\end{figure}
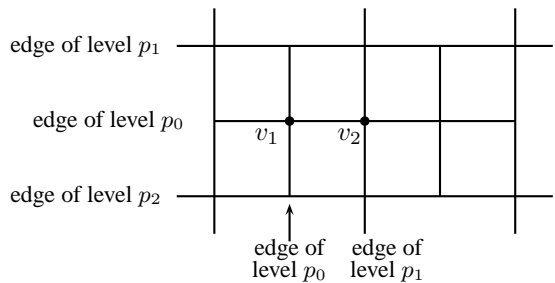
We assign every edge or l-edge in a hierarchical T-mesh with a level
number as same as the level number of the T-mesh where the edge or
l-edge just appears. An l-edge of level $k$ consists of edges of
level $k$. The extension of an edge in the extended T-mesh is
assigned with the same level as its source. The new adding l-edges
in the extended T-mesh are assigned with level 0.

A crossing vertex is assigned with two level numbers, denoted as
$(k_h,k_v)$, corresponding to the level number $k_h$ of the
horizontal l-edge and the level number $k_v$ of the vertical l-edge
where the vertex lies, respectively. $k_h$ and $k_v$ are called
horizontal level and vertical level, respectively, of the crossing
vertex.

For example, in the hierarchical T-mesh $\mathscr{T}$ as shown in
Figure \ref{BEdgeL}, suppose the middle horizontal l-edge is with
level $p_0$, and its upper and lower neighboring l-edges are with
levels $p_1$ and $p_2$, respectively. Here it should be $p_0>p_1$,
$p_0>p_2$. Because $v_1$ is the intersection of two l-edges with
level $p_0$, its level is $(p_0,p_0)$. $v_2$ is the intersection of
a horizontal l-edge of level $p_0$ and a vertical l-edge of level
$p_1$. Hence its level is $(p_0,p_1)$.

\subsubsection{An ordering on interior l-edges}
\label{ordering}

In order to sort the constraints reasonably, we introduce a partial
ordering on interior l-edges in a hierarchical T-mesh. This order
will be used in the proof of Theorem \ref{thm} to facilitate the
rank analysis of the constraints. Before that, we propose the
following definition.

\begin{definition}
\label{cvc1} In a hierarchical T-mesh $\mathscr{T}$, two interior
l-edges are \textbf{continuous} if they intersect in a crossing
vertex of $\mathscr{T}$. An l-edge set $S$ is \textbf{connected} if,
for any two l-edges $\ell_0$ and $\ell_1$ in $S$, there exists a
\textbf{continuous series of l-edges} $l_0, l_1, \ldots, l_k$ in $S$
between $\ell_0$ and $\ell_1$, i.e., $l_0 = \ell_0$, $l_k = \ell_1$,
and $l_i$ and $l_{i+1}$ are continuous for any $i=0, 1, \ldots,
k-1$.
\end{definition}

\begin{figure}[!ht]
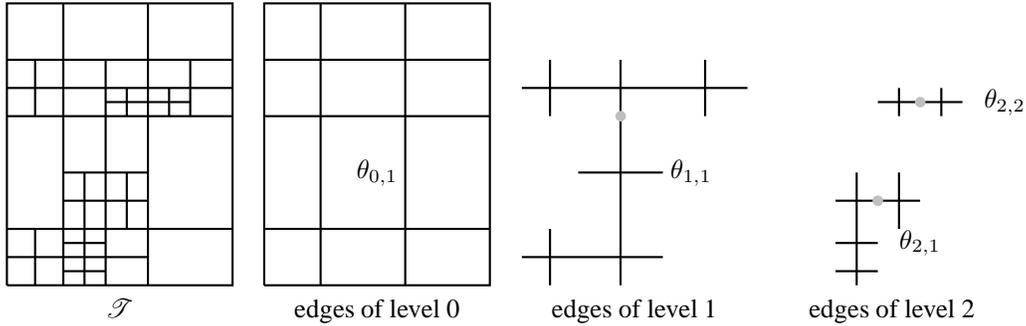

\begin{center}
\psset{unit=0.75cm,linewidth=0.8pt}
\begin{tabular}{cccc}
\pspicture(0,0)(4,5) \psline(0,0)(4,0)(4,5)(0,5)(0,0)
\psline(0,1)(4,1)\psline(0,3)(4,3)\psline(0,4)(4,4)
\psline(1,0)(1,5)\psline(2.5,0)(2.5,5)\psline(1.75,0)(1.75,4)
\psline(0,0.5)(2.5,0.5)\psline(0.5,0)(0.5,1)\psline(2.5,2)(1,2)\psline(0,3.5)(4,3.5)
\psline(0.5,3)(0.5,4)\psline(3.25,3)(3.25,4)
\psline(1.375,0)(1.375,2)\psline(1,0.25)(1.75,0.25)
\psline(1,0.75)(1.75,0.75)\psline(1,1.5)(2.5,1.5)\psline(2.125,1)(2.125,2)
\psline(2.125,3)(2.125,3.5)
\psline(1.75,3.25)(3.25,3.25)\psline(2.875,3)(2.875,3.5)
\endpspicture
& \pspicture(0,0)(4,5) \psline(0,0)(4,0)(4,5)(0,5)(0,0)
\psline(0,1)(4,1)\psline(0,3)(4,3)\psline(0,4)(4,4)
\psline(1,0)(1,5) \psline(2.5,0)(2.5,5) \rput(2,2){$\theta_{0,1}$}
\endpspicture
& \pspicture(0,0)(4,5) \psline(1.75,0)(1.75,4)
\psline(0,0.5)(2.5,0.5)\psline(0.5,0)(0.5,1)\psline(2.5,2)(1,2)\psline(0,3.5)(4,3.5)
\psline(0.5,3)(0.5,4)\psline(3.25,3)(3.25,4)
\rput(3,2){$\theta_{1,1}$} \psset{linewidth=1pt,linecolor=lightgray}
\psdots(1.75,3)
\endpspicture
& \pspicture(0,0)(4,5)
\psline(1.375,0)(1.375,2)\psline(1,0.25)(1.75,0.25)
\psline(1,0.75)(1.75,0.75)\psline(1,1.5)(2.5,1.5)\psline(2.125,1)(2.125,2)
\psline(2.125,3)(2.125,3.5)
\psline(1.75,3.25)(3.25,3.25)\psline(2.875,3)(2.875,3.5)
\rput(2.5,0.75){$\theta_{2,1}$}\rput(4,3.25){$\theta_{2,2}$}
\psset{linewidth=1pt,linecolor=lightgray}
\psdots(2.5,3.25)(1.75,1.5)
\endpspicture \\
$\mathscr{T}$ & edges of level 0  & edges of level 1 & edges of
level 2
\end{tabular}
\caption{A hierarchical T-mesh with its decomposition according to
the edge levels \label{order}}
\end{center}
\end{figure}

Consider a crossing-vertex connected hierarchical T-mesh
$\mathscr{T}$. Fix a level number $k \geqslant 0$. Then all the
l-edges with level $k$ is possibly not connected. See Figure
\ref{order} for an example, where all the edges of level 2 are not
connected. We assume that they form some maximal connected branches.
Denote these branches to be $\theta_{k,i}$, $i=1,\ldots, T_k$. Hence
in the example of Figure \ref{order}, we have $T_0=1$, $T_1=1$, and
$T_2=2$. Because $\mathscr{T}$ is crossing-vertex connected, when
$k>1$, there exists at least one crossing vertex with level number
$(k,j)$ or $(j,k)$ on some l-edge in the branch $\theta_{k,i}$,
where $j < k$. See Figure \ref{order} for examples, where the
specified crossing vertex is shown in light-gray. This crossing
vertex is the intersection between two l-edges of level $k$ and $j$,
where the l-edge of level $j$ is called an \textbf{entering l-edge}
of the branch, which is connected with every l-edge in
$\theta_{k,i}$. The entering l-edge of all the l-edges of level zero
is defined to be any l-edge of level zero.

Fixed an entering l-edge $\ell_{0}$ of $\theta_{k,i}$. For an l-edge
$\ell$ in $\theta_{k,i}$, there exists many continuous series of
l-edges in $\theta_{k,i}$ connecting $\ell$ and $\ell_0$. The number
of the l-edges in a series is called the length of the series, and
the minimal length of all the series connecting $\ell$ and $\ell_0$
is called the \textbf{distance between $\ell$ and $\ell_0$}, denoted
as $\dist(\ell,\ell_0)$. Suppose $e_0 = \ell_0, e_1, \ldots, e_s =
\ell$ is an l-edge series with the minimal length among the
continuous series connecting $\ell$ and $\ell_0$. Then the
intersection point between $e_s = \ell$ and $e_{s-1}$ is a crossing
vertex on $\ell$, which is called a \textbf{characteristic vertex of
the l-edge} $\ell$. Its level is $(k,j)$ or $(j,k)$, where $j
\leqslant k$.

After having selected the entering l-edges for all the connected
branches $\theta_{k,i}$, we introduce a partial ordering $<_1$ on
all the interior l-edges in $\mathscr{T}$. For any two interior
l-edges $\ell_1$ and $\ell_2$ with levels $k_1$ and $k_2$,
respectively, where $\ell_j$ is in the branch $\theta_{k_j, i_j}$,
$j=1,2$, we define $\ell_1 <_1 \ell_2$ if
\begin{enumerate}
\item $k_1 < k_2$, or
\item $k_1 = k_2$ and $i_1 < i_2$, or
\item $k_1 = k_2$ and $i_1 = i_2$, $\dist(\ell_1,\ell_0) < \dist(\ell_2, \ell_0)$,
where $\ell_0$ is the entering l-edge of the connected branch in
which $\ell_1$ and $\ell_2$ lies, since these two l-edges are in the
same connected branch.
\end{enumerate}
This order is not total, because, in Case 3, it is possible that
$\dist(\ell_1,\ell_0) = \dist(\ell_2, \ell_0)$ for two different
interior l-edges.

\subsection{Conversion of constraints}
\label{sec5.2}

In Section \ref{sec4} we have proposed some different versions of
the necessary and sufficient conditions that ensure $\mathcal{I}(g)
\in \overline{\mathbf{S}}(2,2,1,1,\mathscr{T})$ for any $g \in
\overline{\mathbf{S}}(1,1,0,0,\mathscr{T})$. In order to facilitate
the latter analysis, we take the following notation and the
constraints.

For any regular T-mesh $\mathscr{T}$, denote its occupied rectangle
to be $(x_l, x_r)\times (y_b, y_t)$. $\mathscr{T}^c$ is an
associated tensor-product mesh with $\mathscr{T}$. Assume the $y$
coordinates of all the horizontal l-edges in $\mathscr{T}^c$ to be
$y_b=y_0<y_1<\cdots<y_m<y_{m+1}=y_t$, and the $x$ coordinates of all
the vertical l-edges in $\mathscr{T}^c$ to be $x_l = x_0 < x_1 <
\cdots < x_n < x_{n+1}=x_r$. For $i=0,1,\ldots, m+1$, assume the
horizontal l-edges $l^h_{i1}, \ldots, l^h_{i\alpha_i}$ are with $y$
coordinate $y=y_i$. Here $l^h_{i1}, \ldots, l^h_{i\alpha_i}$ are
sorted from left to right. On the other hand, for $j=0,1,\ldots,
n+1$, assume the vertical l-edges $l^v_{j1}, \ldots, l^v_{j\beta_j}$
are with $x$ coordinate $x= x_j$. Here $l^v_{j1}, \ldots,
l^v_{j\beta_j}$ are sorted from bottom to top. It follows that
$\alpha_0=\alpha_{m+1}=\beta_0=\beta_{n+1}=1$,
$\alpha_1+\cdots+\alpha_{m}+ \beta_1+\cdots+\beta_n=E$.

With these notation, according to Lemma \ref{lemm4.3}, we can
allocate the $E$ constraints to the interior l-edges in
$\mathscr{T}^c$ in the following fashion: To any interior horizontal
l-edge $y=y_i$, the corresponding constraints are
\begin{align} & \int_{l^h_{ik}} \frac{\partial}{\partial y}
g(s,y_i-) ds = \int_{l^h_{ik}} \frac{\partial}{\partial y} g(s,y_i+)
ds, \quad k=1,\ldots,\alpha_i. \label{eqn5.1}
\end{align}
To any interior vertical l-edge $x=x_j$, the corresponding
constraints are
\begin{align} &\int_{l^v_{jk}} \frac{\partial}{\partial x} g(x_j-,t)
dt = \int_{l^v_{jk}} \frac{\partial}{\partial x} g(x_j+,t) dt, \quad
k=1,\ldots,\beta_j. \label{eqn5.2}
\end{align}
Furthermore, by applying the similar deduction with Equation
\eqref{eqn4.6}, Equations \eqref{eqn5.1} and \eqref{eqn5.2} are
equivalent with
\begin{align}
(y_{i+1}-y_{i-1})\int_{x^h_{ik0}}^{x^h_{ik1}} g(s,y_i)ds =
(y_{i+1}-y_i)\int_{x^h_{ik0}}^{x^h_{ik1}} g(s,y_{i-1})ds &{}+
(y_i-y_{i-1})
\int_{x^h_{ik0}}^{x^h_{ik1}} g(s,y_{i+1})ds, \nonumber\\
& k=1,\ldots,\alpha_i, \label{eqn5.3}
 \\
(x_{j+1}-x_{j-1})\int_{y^v_{jk0}}^{y^v_{jk1}} g(x_j,t)dt =
(x_{j+1}-x_j)\int_{y^v_{jk0}}^{y^v_{jk1}} g(x_{j-1},t)dt &{}+
(x_j-x_{j-1})
\int_{y^v_{jk0}}^{y^v_{jk1}} g(x_{j+1},t)dt,\nonumber  \\
&  k=1,\ldots,\beta_j, \label{eqn5.4}
\end{align}
respectively, where $x^h_{ik0}$ and $x^h_{ik1}$ are the $x$
coordinates of the two end-points of $l^h_{ik}$, and $y^v_{jk0}$ and
$y^v_{jk1}$ are the $y$ coordinates of the two end-points of
$l^v_{jk}$.

Now we focus on hierarchical T-meshes, on which the corresponding
constraints can be converted into a form to reflect the level
structure of the T-mesh. At first, we introduce two definitions.

In a hierarchical T-mesh $\mathscr{T}$, select any horizontal l-edge
$l$ with level $k>0$. Hence $l$ appears in $\mathscr{T}$ since
$\mathscr{T}^k$. On $l$, there exists one or more crossing vertices
with vertical level $k$. These crossing vertices are the center of
some inserted crossing from $\mathscr{T}^{k-1}$ to $\mathscr{T}^k$.
The l-edge $l$ consists of the horizontal edges of these inserted
crossing. It follows that the vertical edges of the inserted
crossing intersect with two l-edges $l^t$ and $l^b$ in
$\mathcal{T}^{k-1}$. Here we assume $l^b$ with level $k^b$ lies
under $l$, and $l^t$ with level $k^t$ lies above $l$.
\begin{definition}
\label{support} For a horizontal l-edge $l$ with level $k>0$, the
l-edges $l^b$ and $l^t$ are defined as the above description. Then
$l^b$ and $l^t$ are called the \textbf{support l-edges} of $l$. For
an horizontal l-edge $l$ with level 0, its support l-edges are
defined to be the two nearest horizontal l-edges with level 0, which
lie above and under $l$, respectively. For vertical l-edges, the
support l-edges are defined in a similar way.
\end{definition}

It is obvious that, for any horizontal/vertical l-edge $l$, two
vertical/horizontal l-edges, which are through the two end-points of
$l$, intersect with the two support l-edges of $l$, where there is
not other crossing vertex between the intersection points and the
end-points along the two vertical/horizontal l-edges.

\begin{definition}
\label{rank} For the $E$ constraints ensuring $\mathcal{I}(g) \in
\overline{\mathbf{S}}(2,2,1,1,\mathscr{T})$, the number $r$ of its
maximal linearly independent subset is called the \textbf{rank of
the constraints}. Here $E-r$ is called the \textbf{defective rank of
the constraints}.
\end{definition}
In fact, if we specify a set of basis functions of
$\overline{\mathbf{S}}(1,1,0,0,\mathscr{T})$, and undetermine the
coefficients of $g$ with these basis functions, then these
constraints can be represented into the linear combinations of these
coefficients. Hence the rank of the constraints is the rank of the
corresponding coefficient matrix.

In Equations \eqref{eqn5.3} and \eqref{eqn5.4}, a constraint along
an l-edge is represented into the linear combination of three
integrations, which are along the current l-edge and its two
neighboring horizontal/vertical lines, respectively, with the same
integration limits. Here the horizontal/vertical lines share the
same $y$/$x$ coordinates as the two-sided nearest
horizontal/vertical l-edges to the current l-edge. In the following
lemma, we will convert these constraint into a new form, such that
every constraint along an l-edge is a linear combination of three
integration, which are along the current l-edge and its support
l-edges, respectively. With this form, every constraint will involve
undetermined coefficients in a way easy for rank determining.

\begin{lemma}
\label{lemm5.2} Given a hierarchical T-mesh $\mathscr{T}$, select an
interior horizontal l-edge $l^h_i$. Suppose the $y$ coordinate of
$l^h_i$ is $y^h_i$, the $x$ coordinates of its two end-points are
$x^h_{i1}$ and $x^h_{i2}$, and the $y$ coordinates of its support
l-edges are $y^{hb}_i$ and $y^{ht}_i$. Then Equation \eqref{eqn5.3}
holds for all the interior horizontal l-edges if and only if, for
all the interior horizontal l-edges, the following equation holds:
\begin{align}
\label{eqn5.5} (y^{ht}_i - y^{hb}_i)\int_{x^h_{i1}}^{x^h_{i2}}
g(s,y_i^h) ds = (y^{ht}_i - y^{h}_i)\int_{x^h_{i1}}^{x^h_{i2}}
g(s,y_i^{hb}) ds + (y^{t}_i - y^{hb}_i)\int_{x^h_{i1}}^{x^h_{i2}}
g(s,y_i^{ht}) ds.
\end{align}
For interior vertical l-edges, the conclusion is similar.
\end{lemma}

\begin{proof}
Necessity. Suppose the interior horizontal l-edge $l^h$ is with
level $k$, and its two support l-edges are $l^b$ and $l^t$. Between
$l^h$ and $l^b$, there are some other l-edges. Suppose the different
$y$ coordinates these l-edges are $\bar{y}_0<\bar{y}_1< \cdots <
\bar{y}_{\bar{n}}$, where $\bar{y}_0$ and $\bar{y}_{\bar{n}}$
correspond to $l^b$ and $l^h$, respectively. Suppose again the
different $y$ coordinates of the l-edges between $l^h$ and $l^t$ are
$\bar{y}_{\bar{n}}<\bar{y}_{\bar{n}+1}< \cdots <
\bar{y}_{\bar{n}+\bar{m}}$, where $\bar{y}_{\bar{n}+\bar{m}}$
corresponds to $l^t$. According to the definition of the support
l-edges, these horizontal l-edges excluding $l^h$, $l^b$ and $l^t$
are with level greater than $k$. Hence there does not exist one of
them such that its vertical projection onto $l^h$ takes one of the
end-points of $l^h$ as its interior point. Therefore, according to
the constraints in Equation \eqref{eqn5.3} corresponding to these
l-edges, we can conclude that $(\bar{y}_{i-1}, I_{i-1})$,
$(\bar{y}_i,I_i)$, $(\bar{y}_{i+1}, I_{i+1})$ are collinear, where
$i=1,\ldots, \bar{m}+\bar{n}-1$, and
$$ I_i = \int_{x_{i1}^h}^{x_{i2}^h} g(s,\bar{y}_i) ds.
$$
Then $(\bar{y}_0, I_0)$, $(\bar{y}_{\bar{n}}, I_{\bar{n}})$,
$(\bar{y}_{\bar{n}+\bar{m}}, I_{\bar{n}+\bar{m}})$ are collinear,
which means the corresponding constraint in Equation \eqref{eqn5.5}
holds.

Sufficiency. We prove it in an inductive fashion from the highest
level to the lowest level. We will show that, for any given level
$k_0$, if the constraints as defined in Equation \eqref{eqn5.5} hold
corresponding to the l-edges of level $k \geqslant k_0$, then the
constraints in Equation \eqref{eqn5.3} corresponding to all the
l-edges of level $k_0$ hold as well.

Suppose the maximal level number in the hierarchical T-mesh
$\mathscr{T}$ is $M$. For an interior horizontal l-edge with level
$M$, since there do not exist other l-edges between the current
l-edge and its support l-edges, the corresponding constraint defined
in Equation \eqref{eqn5.3} is the same as one defined in Equation
\eqref{eqn5.5}.

Now we assume that, for all the interior horizontal l-edges with
level greater than $k$, if the corresponding constraints in Equation
\eqref{eqn5.5} hold, then the corresponding constraints in Equation
\eqref{eqn5.3} hold as well. Select an arbitrary interior horizontal
l-edge $l^h$ with level $k$, whose support l-edges are $l^b$ and
$l^t$. The two-sided nearest l-edges of $l^h$ are with $y$
coordinates $h_0$ and $h_1$. Since the level numbers of the other
horizontal l-edges between $l^h$ and $l^b$ are greater than $k$,
according to the inductive assumption, the integration values along
these horizontal lines with the integration interval
$[x^h_{i1},x^h_{i2}]$ are collinear with respect to their $y$
coordinates. Especially, the corresponding integration values along
$l^h$, $l^b$ and $y = h_0$ are collinear with respect to their $y$
coordinates. Similarly, the integration values along $l^h$, $l^t$
and $y = h_1$ are collinear with respect to their $y$ coordinates as
well. Hence the integration values along $l^h$, $y=h_0$ and $y =
h_1$ are collinear with respect to their $y$ coordinates. This
finishes the proof of the sufficiency.
\end{proof}

In the next subsection we will specify a proper set of basis
functions of $\overline{\mathbf{S}}(1,1,0,0,\mathscr{T})$ such that
the coefficients of $g$ under these basis functions appear in these
constraints in a regular form.

\subsection{Hierarchical basis functions of $\overline{\mathbf{S}}(1,1,0,0,\mathscr{T})$}
\label{hbf}

In Subsection \ref{sec3.4} an approach has been proposed to specify
a set of basis functions of
$\overline{\mathbf{S}}(1,1,0,0,\mathscr{T})$ with many good
properties. Now we specify a new set of basis functions for
$\overline{\mathbf{S}}(1,1,0,0,\mathscr{T})$, which does not poss
the unity-partition property. But under this set, the former
constraints in Lemma \ref{lemm5.2} will appear in a regular form,
which facilitates the rank determination of the constraints.

The new set of basis functions is defined level by level when
forming the hierarchical T-mesh. Every basis function is associated
with a crossing vertex. At the level 0, we consider the level 0
T-mesh $\mathscr{T}^0$, and introduce the functions to be standard
bilinear tensor-product B-splines. Hence every function can be
associated with a crossing vertex in $\mathscr{T}^0$ in a way that
the function reaches one at the crossing vertex, and zero at all the
other crossing vertices of $\mathscr{T}^0$. Use the set
$\mathcal{B}^0$ to denote all these functions. Suppose the current
level number is $k\geqslant 1$. Consider a new-coming crossing
vertex in this level.
\begin{figure}[!htp]
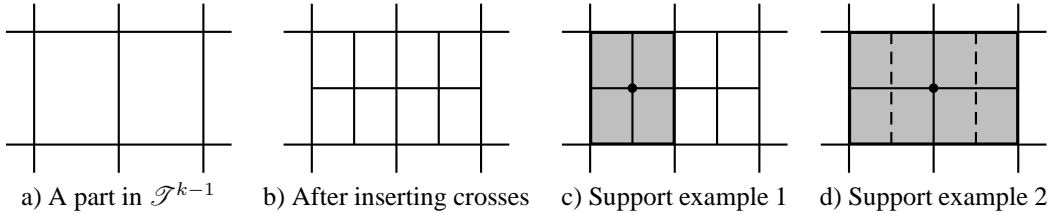

\begin{center}
\psset{unit=0.75cm,linewidth=0.8pt} \noindent\begin{tabular}{cccc}
\pspicture(0,0)(4,3) \psline(0,0.5)(4,0.5)\psline(0,2.5)(4,2.5)
\psline(0.5,0)(0.5,3)\psline(2,0)(2,3)\psline(3.5,0)(3.5,3)
\endpspicture &
\pspicture(0,0)(4,3) \psline(0,0.5)(4,0.5)\psline(0,2.5)(4,2.5)
\psline(0.5,0)(0.5,3)\psline(2,0)(2,3)\psline(3.5,0)(3.5,3)
\psline(0.5,1.5)(3.5,1.5)\psline(1.25,0.5)(1.25,2.5)\psline(2.75,0.5)(2.75,2.5)
\endpspicture
& \pspicture(0,0)(4,3)
\psframe[fillstyle=solid,fillcolor=lightgray](0.5,0.5)(2,2.5)
\psline(0,0.5)(4,0.5)\psline(0,2.5)(4,2.5)
\psline(0.5,0)(0.5,3)\psline(2,0)(2,3)\psline(3.5,0)(3.5,3)
\psdots(1.25,1.5)
\psline(0.5,1.5)(3.5,1.5)\psline(1.25,0.5)(1.25,2.5)\psline(2.75,0.5)(2.75,2.5)
\endpspicture &
\pspicture(0,0)(4,3)
\psframe[fillstyle=solid,fillcolor=lightgray](0.5,0.5)(3.5,2.5)
\psline(0,0.5)(4,0.5)\psline(0,2.5)(4,2.5)
\psline(0.5,0)(0.5,3)\psline(2,0)(2,3)\psline(3.5,0)(3.5,3)
\psdots(2,1.5)
\psline(0.5,1.5)(3.5,1.5)\psline[linestyle=dashed](1.25,0.5)(1.25,2.5)
\psline[linestyle=dashed](2.75,0.5)(2.75,2.5)
\endpspicture \\
a) A part in $\mathscr{T}^{k-1}$ & b) After inserting crosses & c)
Support example 1 & d) Support example 2
\end{tabular}
\end{center}
\caption{The support of the hierarchical basis functions
\label{support}}
\end{figure}

\begin{enumerate}
\item If its level is $(k,k)$, then the crossing vertex must be the
center of an inserted cross into a cell $c$ of $\mathcal{T}^{k-1}$.
Hence a function can be defined associated with the crossing vertex
such that the function reaches one at the vertex and its support is
$c$. See Figure \ref{support}.c) for an example, where the filled
region with light-gray is the support of the specified function
associated with the new crossing vertex labeled with $\bullet$. Here
the discontinuous of the derivatives $\partial/\partial x$ and
$\partial/\partial y$ appears on the edges of the inserted cross.

\item Otherwise, the crossing vertex is the middle point of an edge
$e$ in $\mathscr{T}^{k-1}$, where $e$ is the common edge of two
neighboring cells $c_1$ and $c_2$, each of which is subdivided by
inserting a cross from level $k-1$ to $k$. Then a function can be
defined such that the function reaches one at the current crossing
vertex, its support is exactly $c_1 \cup c_2$, and its derivative
discontinuous in the support lies only on the lines through the
current crossing vertex. See Figure \ref{support}.d) for an example,
where the derivatives $\partial/\partial x$ and $\partial/\partial
y$ are continuous along the dashed edges.
\end{enumerate}

All the functions introduced at level $k$ are denoted to be
$\mathcal{B}^k$.

\begin{lemma}
\label{lemm5.3} Suppose the maximal level number of the hierarchical
T-mesh $\mathscr{T}$ is $M$. Define $\mathcal{B}=\bigcup_{k=0}^{M}
\mathcal{B}^k$. Then $\mathcal{B}$ is a set of basis functions of
$\overline{\mathbf{S}}(1,1,0,0,\mathscr{T})$.
\end{lemma}

\begin{proof}
According to the definition of the functions in $\mathcal{B}^k$, it
follows that the number of the functions in $\mathcal{B}$ is exactly
the number of crossing vertices in $\mathscr{T}$. Moreover, these
functions are in the spline space
$\overline{\mathbf{S}}(1,1,0,0,\mathscr{T})$. Therefore, in order to
prove that $\mathcal{B}$ forms a set of basis functions of
$\overline{\mathbf{S}}(1,1,0,0,\mathscr{T})$, one just needs to show
the functions in $\mathcal{B}$ are linearly independent.

Assume the region occupied by $\mathscr{T}$ is $\Omega$. Let
$\mathcal{B}^k =\{b^k_1(x,y), \ldots, b^k_{n_k}(x,y)\}$. Suppose a
set of coefficients of $\alpha^k_i$ ensure that
$$f(x,y):= \sum_{k=0}^{M} \sum_{i=0}^{n_k} \alpha^k_i b^k_i(x,y) = 0,
\quad (x,y) \in \Omega. $$ Now we will prove that $\alpha^k_i=0$ for
any $i$ and $k$.

Consider the values of $f$ at every crossing vertex $v^0_i$ with
level $(0,0)$. Since $f \equiv 0$, it follows that $f(v^0_i)=0$. On
the other hand, the function $b^0_i(x,y)$ in $\mathcal{B}^0$
associated with $v^0_i$ is one at $v^0_i$, and all the other
functions vanish at $v^0_i$. Hence $f(v^i_0)=\alpha^0_i
b^0_i(v^i_0)=\alpha^0_i$, i.e., $\alpha^0_i=0$. Hence, we have
$$f(x,y)= \sum_{k=1}^{M} \sum_{i=0}^{n_k} \alpha^k_i b^k_i(x,y). $$

Suppose we have proved that $\alpha_i^j=0$, $i=0,\ldots, n_j$,
$j=0,1,\ldots,k-1$. Now we consider the functions and their
coefficients introduced at level $k$. Each of such the functions is
associated with a crossing vertex. We first consider the functions
which are associated with crossing vertices whose vertical and
horizontal level numbers are different. On such a crossing vertex,
all the functions are zero except the associated function with the
current crossing vertex. Hence the corresponding coefficient is
zero. Then we consider the other functions, which are associated
with crossing vertices with the same vertical and horizontal level
numbers. One can conclude that their coefficients are zero as well
in a similar way. Hence all the coefficients are zero, and the
functions in $\mathscr{B}$ are linearly independent. Therefore they
form a set of basis functions of
$\overline{\mathbf{S}}(1,1,0,0,\mathscr{T})$.
\end{proof}

W will call this set of basis functions to be \textbf{the
hierarchical basis functions} of
$\overline{\mathbf{S}}(1,1,0,0,\mathscr{T})$.

\subsection{Concretion of the constraints}

For a given hierarchical T-mesh $\mathscr{T}$, suppose the
hierarchical basis functions of
$\overline{\mathbf{S}}(1,1,0,0,\mathscr{T})$ are
$$\{b^k_i(x,y)\}=\{c_j(x,y)\}_{j=1}^{V^+},$$
where $b^k_i(x,y)$ is associated with a crossing vertex appearing in
$\mathscr{T}$ since level $k$.

Represent $g$ as \begin{equation} \label{eqn5.6} g(x,y) =
\sum_{j=1}^{V^+} \alpha_j c_j(x,y).
\end{equation}
Since there is a one-to-one mapping between the basis functions and
the crossing vertices, the coefficient of a basis function is called
also the \textbf{coefficient of the corresponding crossing vertex}.

Now we state the characteristic of the constraints after
substituting the representation of $g$ as defined in Equation
\eqref{eqn5.6} into Equation \eqref{eqn5.5} which ensures
$\mathcal{I}(g) \in \overline{\mathbf{S}}(2, 2, 1,1,\mathscr{T})$.
At first we consider the constraints corresponding to the interior
l-edges of level 0.

\begin{proposition}
\label{prop5.5} After substituting Equation \eqref{eqn5.6} into
Equation \eqref{eqn5.5}, and then taking a proper transformation,
all the constraints corresponding to the interior horizontal l-edges
of level 0 are in such a form that the nonzero terms in a constraint
consist of those associated with the crossing vertices on the same
l-edge.
\end{proposition}

\begin{proof}
Each of interior horizontal l-edges with level 0 must traverse the
whole mesh. Its support l-edges are with level 0 as well. Over the
T-mesh $\mathscr{T}^0$, we apply Lemma \ref{lemm4.2} and obtain that
all the constraints corresponding to the interior horizontal l-edges
of level 0 are equivalent with that the integration of $g$ along
each interior horizontal l-edge is zero. For every l-edge of level
0, among all the coefficients, only those of the crossing vertices
on the current l-edge are nonzero. Hence the integration can be
represented into a linear combination of the coefficients of the
crossing vertices on the current l-edge.
\end{proof}

Then we consider the constraints corresponding to the interior
l-edges of level greater than $0$.

\begin{proposition}
\label{prop5.6} Suppose the current interior horizontal l-edge is
with level $k>0$. After substituting Equation \eqref{eqn5.6} into
Equation \eqref{eqn5.5}, the nonzero coefficients of the basis
functions in the corresponding constraint consist of two parts as
follows:
\begin{enumerate}
\item[A.] the coefficients of the crossing vertices on the current l-edge;
\item[B.] the possible coefficients of the crossing vertices with
horizontal level less than $k$ and vertical level greater than $k$.
\end{enumerate}
\end{proposition}

\begin{proof}
Select an interior horizontal l-edge $l^h$ with level $k>0$. Suppose
its support l-edges are $l^b$ and $l^t$ with levels $k_1$ and $k_2$,
respectively. Assume the $x$ coordinates of the two end-points of
$l^h$ are $x^h_0$ and $x^h_1$, and the vertical l-edges through the
two end-points are $l^v_l$ and $l^v_r$. Then the two vertical
l-edges $l^v_l$ and $l^v_r$ intersect with $l_b$ and $l_t$. Consider
the crossing vertices whose associating basis functions are nonzero
on $l^h$, $l^b|_{[x^h_0,x^h_1]}$ or $l^t|_{[x^h_0,x^h_1]}$. These
crossing vertices can be classified into the following cases with
respect to their level $(k^h_0, k^h_1)$:
\begin{enumerate}
\item $k^h_0 < k$:
\begin{enumerate}
\item $k^h_1 \leqslant k$. Because, for any $x \in [x^h_0,x^h_1]$,
the corresponding three points on $l^h$, $l^t$, and $l^b$ with the
horizontal coordinates $x$ appear simultaneously in a cell of
$\mathscr{T}^{k-1}$ (including its boundary), it follows that for
the current basis function $b(x,y)$, Equation \eqref{eqn5.5} holds
as $g = b(x,y)$. This means that the coefficients of $b(x,y)$ does
not appear in the constraint after simplification.

\item $k^h_1
>k$. This types of coefficients can appear in the constraint.
\end{enumerate}

\item $k^h_0 =k$: This type of crossing vertex has no contribution to
the constraints, unless it lies on $l^h$.

\item $k^h_0 > k$: The associating basis function with this type of crossing vertex
vanishes on $l^h$, $l^b$ and $l^t$ between $x^h_0$ and $x^h_1$.
Hence the corresponding coefficient does not appear in the
constraint.
\end{enumerate}
Therefore the nonzero coefficients come in Cases 1(b) and 2, which
correspond to Cases B and A in the proposition description,
respectively.
\end{proof}

We can make a similar classification on the coefficients' appearance
in the constraints corresponding to interior vertical l-edges.

After these classifications, one knows that the coefficient of a
crossing vertex with level $(k_1,k_2)$ can appear in the following
three types of places:
\begin{enumerate}
\item The constraints corresponding to the horizontal l-edge through the vertex;
\item The constraints corresponding to the vertical l-edge through the vertex;
\item If $k_1<k_2-1$ (or $k_1-1 > k_2$), it may appear in the constraints
corresponding to horizontal (or vertical) l-edges with level less
than $k_2$ (or $k_1$). If $k_1=k_2$ or $k_2 \pm 1$, this situation
does not happen.
\end{enumerate}
The involved l-edges in the first two cases are called
\textbf{naturally appearing l-edges of the coefficient}. The
involved l-edges in the third case are called \textbf{unnaturally
appearing l-edges of the coefficient}. For simplify, a constraint
corresponding to a horizontal/vertical l-edge of level $k$ is simply
called the \textbf{horizontal/vertical constraint of level $k$} in
the following. For two given interior l-edges $\ell_1$ and $\ell_2$,
their corresponding constraints are $c_1$ and $c_2$, respectively.
If $\ell_1<_1 \ell_2$, then we define $c_1 <_1 c_2$ as well.

\subsection{Dimension theorem over crossing-vertex connected hierarchical T-meshes}
\label{sec5.5}

With these preparations, now we state and prove the dimension
theorem of biquadratic spline spaces over crossing-vertex connected
hierarchical T-meshes.

\begin{theorem}
\label{thm} Over a crossing-vertex connected hierarchical T-mesh
$\mathscr{T}$ or its extension associated with
$\mathbf{S}(2,2,1,1,\mathscr{T})$, where $V^+>0$, it follows that
$$ \dim \overline{\mathbf{S}}(2,2,1,1,\mathscr{T}) = V^+ - E + 1. $$
\end{theorem}

\begin{proof}
Recall some facts first. For any $g \in
\overline{\mathbf{S}}(1,1,0,0,\mathscr{T})$, in order to ensure that
$\mathcal{I}(g) \in \overline{\mathbf{S}}(2,2,1,1,\mathscr{T})$, the
constraints corresponding to every interior horizontal l-edges
(defined in Equation \eqref{eqn5.5}) and vertical l-edges should be
satisfied. Here we apply the hierarchical basis functions
$\mathcal{B}$ of $\overline{\mathbf{S}}(1,1,0,0,\mathscr{T})$. These
constraints can be represented into linear combinations of the
coefficients of $g$ under the basis functions $\mathcal{B}$. The
coefficients appear in these constraints as stated in Propositions
\ref{prop5.5} and \ref{prop5.6}.

Suppose the l-edges of level $k$ in $\mathscr{T}$ form some
connected branch $\theta_{k,i}$, $i=1,2, \ldots, T_k$, and the
entering l-edges for all the branches $\theta_{k,i}$ have been
selected, denoted as $\ell_{k,i}$. Specially, all the interior
l-edges of level zero just form a connected branch, and its entering
l-edge is selected to any one interior l-edge of level zero. Then we
can introduce a partial ordering $<_1$ on the interior l-edges and
the corresponding constraints as stated in Subsection
\ref{ordering}.

At the beginning, all the constraints are linearly dependent because
Equation \eqref{lll} shows a linear combination of these constraints
with result zero. In this linear combination, each of the
coefficients of level zero is nonzero. Hence the rank of the
constraints remains unchanged after deleting any one constraint of
level zero. Without loss of generality, we assume that the deleted
constraint is corresponding to the entering l-edge $\ell_{0,1}$ of
$\theta_{0,1}$. Then we focus on the remaining constraints and we
will show that they are linear independent.

Sort all the remaining constraints into a non-decreasing series
$C_1, C_2, \cdots, C_T$ according to the ordering $<_1$, where $T =
E-1$, and sort the coefficients $\{\alpha_i\}$ in Equation
\eqref{eqn5.6} into a series $\beta_1, \beta_2, \ldots, \beta_{V^+}$
such that $\beta_i$ is a characteristic vertex of the l-edge whose
corresponding constraint is $C_i$, $i=1, \ldots, T$. The rest
variables $\beta_{T+1}, \ldots, \beta_{V^+}$ are arranged randomly.
Then we can write these constraints into the following matrix form:
$$ \left( C_1, C_2, \ldots, C_{T}
\right)^T = \mathbf{M} \left( \beta_1, \beta_2, \ldots, \beta_{V^+}
\right)^T,
$$
where $\mathbf{M} = (m_{ij})$. Because the characteristic vertex of
an l-edge $\ell$ with level $k$ is with level $(k,j)$ or $(j,k)$,
where $j \leqslant k$, it follows that, according to Propositions
\ref{prop5.5} and \ref{prop5.6}, the coefficient $\beta_i$ does not
appear in the constraints $C_{i+1}, \ldots, C_T$, where $i=1,
\ldots, T-1$. Hence $m_{ij}=0$, $i>j$. On the other hand, the matrix
$\mathbf{M}$ is with $m_{ii} \neq 0$, $i=1, \ldots, T$. This means
that the matrix $\mathbf{M}$ is with full row rank, i.e.,
$\mathop{\mathrm{rank}} \mathbf{M} = T$. Therefore the dimension of
the spline space $\overline{\mathbf{S}}(2,2,1,1,\mathscr{T})$ is
$V^+-T = V^+ - E + 1$. This completes the proof of the theorem.
\end{proof}

\subsection{General Dimension Theorem}
\label{gdt}

In this subsection we consider the dimension formula of
$\overline{\mathbf{S}}(2,2,1,1,\mathscr{T})$, where $\mathscr{T}$ is
a general hierarchical T-mesh.

At first, we discuss how to divide a general hierarchical T-mesh
into the union of some crossing-vertex connected hierarchical
T-meshes. Suppose the given hierarchical T-mesh is $\mathscr{T}$,
and let $c_i$, $i=1,\ldots, C$, be all the isolated subdivided cells
with level $k > 0$ in forming $\mathscr{T}$. Then the subdivision
happening in $c_i$ to form $\mathscr{T}$ will form as well a
hierarchical T-mesh, denoted as $\mathscr{U}_i$. Here
$\mathscr{U}_i$ occupies the same region as the cell $c_i$. In the
following we will apply a T-mesh operation `$\backslash$'. Suppose
$\mathscr{T}_i$, $i=0, 1, \ldots, k$, are hierarchical T-meshes.
Then $\mathscr{T}_0 \backslash \{ \mathscr{T}_1,
\ldots,\mathscr{T}_k\}$ is a new T-mesh which consists of edges and
vertices in $\mathscr{T}_0$ but not in the interior of
$\mathscr{T}_i$, $i=1,\ldots,k$. Let $\mathscr{U}_0 = \mathscr{T}$,
and
$$ \mathscr{V}_i = \mathscr{U}_i \backslash \{ \mathscr{U}_j: \mathscr{U}_j \mbox{
is a submesh of } \mathscr{U}_i, j =1, \ldots,C, j \neq i \},
i=0,1,\ldots, C.$$ Then it is easy to verify that $\mathscr{V}_i$ is
a crossing-vertex connected hierarchical T-mesh. Here
$\mathscr{V}_0$ comes from $\mathscr{T}$, and the other
$\mathscr{V}_i$'s come from the isolated subdivided cells with level
$k>0$. Then it follows that $\mathscr{T}$ can be seen as the
disjoint union of $\mathscr{V}_i$, $i=0,1,\ldots, C$.

\begin{lemma}
\label{lem5.11} Suppose the hierarchical T-mesh $\mathscr{T}$ and
its disjoint union of $\mathscr{V}_i$, $i=0,1,\ldots, C$, are
defined as before. Then it follows that
\begin{equation}
\label{eqn5.7} \overline{\mathbf{S}}(2,2,1,1,\mathscr{T}) =
\bigoplus_{i=0}^C \overline{\mathbf{S}}(2,2,1,1,\mathscr{V}_i).
\end{equation}
\end{lemma}
\begin{proof}
At first, we prove that the intersection of any two different
subspaces among $\overline{\mathbf{S}}(2,2,1,1,\mathscr{V}_i)$,
$i=0,1, \ldots, C$ are $\{0\}$. Therefore we can define the direct
sum of these subspaces. Suppose we select the submeshes
$\mathscr{V}_{i_0}$ and $\mathscr{V}_{i_1}$, where $i_0 \neq i_1$.
The submeshes $\mathscr{U}_{i_0}$ and $\mathscr{U}_{i_1}$ are
defined as before. If any one of $\mathscr{U}_{i_0}$ and
$\mathscr{U}_{i_1}$ is not a submesh of the other one, then the
regions occupied by these them are disjoint, which follows that the
intersection of the subspaces over $\mathscr{V}_{i_0}$ and
$\mathscr{V}_{i_1}$ is just $\{0 \}$. Otherwise, we assume, without
loss of generality, that $\mathscr{U}_{i_0}$ is a submesh of
$\mathscr{U}_{i_1}$. Then the region occupied by $\mathscr{U}_{i_0}$
is inside a cell of $\mathscr{U}_{i_1}$. Denote the cell to be $c$.
In $\mathscr{V}_{i_1}$, the zero function is a unique function whose
function values and two partial derivatives of order one are zero
along the boundary of $c$. Hence it follows that the intersection of
the subspaces over $\mathscr{V}_{i_0}$ and $\mathscr{V}_{i_1}$ is
 $\{0 \}$ as well.

It is obvious that
\begin{equation}
\label{eqn5.8} \overline{\mathbf{S}}(2,2,1,1,\mathscr{T}) \supset
\bigoplus_{i=0}^C \overline{\mathbf{S}}(2,2,1,1,\mathscr{V}_i).
\end{equation}
On the other hand, for any $f \in
\overline{\mathbf{S}}(2,2,1,1,\mathscr{T})$, one can construct its
component in each subspace as follows. For any $j=0,1, \ldots, C$,
we can arrange all the meshes $\{\mathscr{U}_i\}_{i=1}^C$, each of
which takes $\mathscr{U}_j$ as a submesh, in an ascending chain as
follows:
$$ \mathscr{U}_j = \mathscr{U}_{i_j} \subset \mathscr{U}_{i_j-1}
\subset \cdots \subset \mathscr{U}_{i_0} = \mathscr{U}_0.$$

Since $\mathscr{V}_0$ comes from $\mathscr{U}_0 = \mathscr{T}$ by
deleting the subdivisions in some isolated subdivided cells, we can
define a new function $\mathbb{P}_0 f$ which meets $f$ with order
one along all the edges in $\mathscr{V}_0$. Then $f - \mathbb{P}_0
f$ vanishes out of those isolated subdivided cells of
$\mathscr{U}_0$. We define recursively that $\mathbb{P}_{j} f$ is a
function in $\overline{\mathbf{S}}(2,2,1,1,\mathscr{V}_j)$ which
meets the function $f - \sum_{k=i_0}^{i_j-1} \mathbb{P}_k f$ with
order one along all the edges in $\mathscr{V}_j$. Then $f -
\sum_{k=i_0}^{i_j} \mathbb{P}_k f$ vanishes in $\mathscr{V}_j$
except in the isolated subdivided cells of $\mathscr{U}_j$.

According to the definition of $\mathbb{P}_k$, it follows that $f -
\sum \mathbb{P}_k f$ vanishes everywhere in $\mathscr{T}$. Hence $f
\in \bigoplus_{i=0}^C \overline{\mathbf{S}}(2,2,1,1,\mathscr{V}_i)$,
which means that
\begin{equation}
\label{eqn5.9} \overline{\mathbf{S}}(2,2,1,1,\mathscr{T}) \subset
\bigoplus_{i=0}^C \overline{\mathbf{S}}(2,2,1,1,\mathscr{V}_i).
\end{equation}
Combining Equations \eqref{eqn5.8} and \eqref{eqn5.9} together, one
gets Equation \eqref{eqn5.7}.
\end{proof}

With this lemma, we have the following theorem about the dimension
formula of the spline spaces over general hierarchical T-meshes.

\begin{theorem}
Suppose $\mathscr{T}$ is a hierarchical T-meshes with $\delta -1$
isolated subdivided cells. Then
$$ \dim \overline{\mathbf{S}}(2,2,1,1,\mathscr{T}) = V^+ - E + \delta. $$
\end{theorem}

\begin{proof}
Suppose $\mathscr{U}_j$ and $\mathscr{V}_j$, $j=0, 1, \ldots, C$ are
defined as in the proof of Lemma \ref{lem5.11}. Then it follows that
$C = \delta -1$ and, according to Lemma \ref{lem5.11},
$$ \dim \overline{\mathbf{S}}(2,2,1,1,\mathscr{T}) = \sum_{i=0}^C
\dim \overline{\mathbf{S}}(2,2,1,1,\mathscr{V}_i). $$ Assume in
$\mathscr{V}_i$ there are $V_i^+$ crossing vertices and $E_i$
interior l-edges. Then
$$ \dim \overline{\mathbf{S}}(2,2,1,1,\mathscr{T}) = \sum_{i=0}^C
(V_i^+-E_i+1). $$ Since any two different meshes among $V_i$ do not
share any common crossing vertices and interior l-edges, it follows
that
$$ \sum_{i=0}^C (V_i^+-E_i+1) = V^+-E+C+1 = V^+-E+\delta. $$
Hence we have
$$ \dim \overline{\mathbf{S}}(2,2,1,1,\mathscr{T}) = V^+ - E + \delta. $$
\end{proof}

\subsection{Some Notes on Construction of Basis Functions}

After obtaining the dimension formulae of biquadratic spline spaces
over hierarchical T-meshes, we naturally consider how to construct
their basis functions with some good properties as stated for
bilinear basis functions in Section \ref{sec3.4}.

To do so, we fist need to make clear of the topological meaning of
$V^+-E+\delta$.

\begin{definition}
Given a hierarchical T-mesh $\mathscr{T}$, one can construct a graph
$\mathscr{G}$ by keeping all the crossing vertices and the line
segments with two endpoints being crossing vertices, and removing
all the other vertices and the edges in $\mathscr{T}$. $\mathscr{G}$
is called the \textbf{crossing-vertex-relationship graph} (CVR graph
for short) of $\mathscr{T}$.
\end{definition}

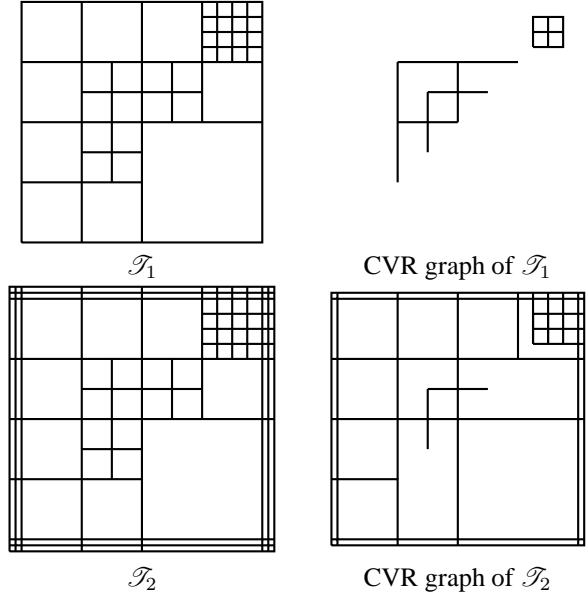
\begin{figure}[!ht]
\begin{center}
\psset{unit=0.8cm,linewidth=0.8pt}
\begin{tabular}{c@{\hspace*{1cm}}c}
\begin{pspicture}(0,0)(4,4)
\psline(0,0)(0,4)(4,4)(4,0)(0,0)
\psline(0,1)(2,1)\psline(1,1.5)(2,1.5)\psline(0,2)(4,2)\psline(1,2.5)(3,2.5)\psline(0,3)(4,3)
\psline(1,0)(1,4)\psline(1.5,1)(1.5,3)\psline(2,0)(2,4)\psline(2.5,2)(2.5,3)\psline(3,2)(3,4)
\psline(3,3.25)(4,3.25)\psline(3,3.5)(4,3.5)\psline(3,3.75)(4,3.75)
\psline(3.25,3)(3.25,4)\psline(3.5,3)(3.5,4)\psline(3.75,3)(3.75,4)
\end{pspicture} &
\begin{pspicture}(0,0)(4,4)
\psline(1,2)(2,2)\psline(1.5,2.5)(2.5,2.5)\psline(1,3)(3,3)
\psline(1,1)(1,3)\psline(1.5,1.5)(1.5,2.5)\psline(2,2)(2,3)
\psline(3.25,3.25)(3.75,3.25)(3.75,3.75)(3.25,3.75)(3.25,3.25)
\psline(3.25,3.5)(3.75,3.5)\psline(3.5,3.25)(3.5,3.75)
\end{pspicture}\\
$\mathscr{T}_1$  & CVR graph of $\mathscr{T}_1$
\\[2mm] \begin{pspicture}(0,0)(4,4)
\psline(0,0)(0,4)(4,4)(4,0)(0,0)\psline(-0.1,-0.1)(-0.1,4.1)(4.1,4.1)(4.1,-0.1)(-0.1,-0.1)
\psline(-0.2,-0.2)(-0.2,4.2)(4.2,4.2)(4.2,-0.2)(-0.2,-0.2)
\psline(-0.1,-0.2)(-0.1,-0.1)\psline(-0.2,-0.1)(-0.1,-0.1)
\psline(0,-0.2)(0,0)\psline(-0.2,0)(0,0)\psline(-0.2,4)(0,4)(0,4.2)
\psline(-0.2,4.1)(-0.1,4.1)(-0.1,4.2)\psline(4.2,4)(4,4)(4,4.2)\psline(4.2,4.1)(4.1,4.1)(4.1,4.2)
\psline(4,-0.2)(4,0)(4.2,0)\psline(4.1,-0.2)(4.1,-0.1)(4.2,-0.1)
\psline(0,1)(2,1)\psline(1,1.5)(2,1.5)\psline(0,2)(4,2)\psline(1,2.5)(3,2.5)\psline(0,3)(4,3)
\psline(1,0)(1,4)\psline(1.5,1)(1.5,3)\psline(2,0)(2,4)\psline(2.5,2)(2.5,3)\psline(3,2)(3,4)
\psline(3,3.25)(4,3.25)\psline(3,3.5)(4,3.5)\psline(3,3.75)(4,3.75)
\psline(3.25,3)(3.25,4)\psline(3.5,3)(3.5,4)\psline(3.75,3)(3.75,4)
\psline(1,0)(1,-0.2)\psline(2,-0.2)(2,0)\psline(4,2)(4.2,2)\psline(4,3)(4.2,3)
\psline(4,3.25)(4.2,3.25)\psline(4,3.5)(4.2,3.5)\psline(4,3.75)(4.2,3.75)
\psline(3.25,4)(3.25,4.2)\psline(3.5,4)(3.5,4.2)\psline(3.75,4)(3.75,4.2)
\psline(3,4)(3,4.2)\psline(2,4)(2,4.2)\psline(1,4)(1,4.2)\psline(-0.2,1)(0,1)
\psline(-0.2,3)(0,3)\psline(-0.2,2)(0,2)
\end{pspicture} &
\begin{pspicture}(0,0)(4,4)
\psline(-0.1,-0.1)(-0.1,4.1)(4.1,4.1)(4.1,-0.1)(-0.1,-0.1)
\psline(0,0)(0,4)(4,4)(4,0)(0,0)\psline(-0.1,0)(0,0)(0,-0.1)\psline(-0.1,4)(0,4)(0,4.1)
\psline(4.1,4)(4,4)(4,4.1)\psline(4,-0.1)(4,0)(4.1,0)\psline(1,-0.1)(1,1)(-0.1,1)
\psline(2,-0.1)(2,2)(4.1,2)\psline(4.1,3)(3,3)(3,4.1)
\psline(1,2)(2,2)\psline(1.5,2.5)(2.5,2.5)\psline(1,3)(3,3)
\psline(1,1)(1,3)\psline(1.5,1.5)(1.5,2.5)\psline(2,2)(2,3)
\psline(3.25,3.25)(3.75,3.25)(3.75,3.75)(3.25,3.75)(3.25,3.25)
\psline(3.25,3.5)(3.75,3.5)\psline(3.5,3.25)(3.5,3.75)\psline(3.75,3.25)(4.1,3.25)
\psline(3.75,3.5)(4.1,3.5)\psline(4.1,3.75)(3.75,3.75)(3.75,4.1)\psline(3.5,3.25)(3.5,4.1)
\psline(3.25,3.25)(3.25,4.1)\psline(2,3)(2,4.1)\psline(1,4.1)(1,3)(-0.1,3)\psline(1,2)(-0.1,2)
\end{pspicture} \\[2mm]
$\mathscr{T}_2$  & CVR graph of $\mathscr{T}_2$
\end{tabular}
\caption{Two examples of CVR graphs.\label{cvrg}}
\end{center}
\end{figure}
See Figure \ref{cvrg} for two examples. Here $\mathscr{T}_1$ has an
isolated subdivided cell at its right-top part. Hence $\delta = 2$.
Its CVR graph has two disconnected parts. $\mathscr{T}_2$ is an
extended T-mesh of $\mathscr{T}_1$ with respect to the spline space
${\mathbf{S}}(2,2,1,1,\mathscr{T}_1)$. $\mathscr{T}_2$ has no
isolated subdivided cell. The corresponding CVR graph is connected.
The following theorem states the relationship between $V^+ - E +
\delta$ in $\mathscr{T}$ and the cell number $F_\mathscr{G}$ in
$\mathscr{G}$.

\begin{theorem}
\label{hint} Given a hierarchical T-mesh $\mathscr{T}$ with $V^+$
crossing vertices, $E$ interior l-edges and $\delta-1$ isolated
subdivided cells, suppose there are $F_\mathscr{G}$ cells in its CVR
graph $\mathscr{G}$. Then it follows that \begin{equation} V^+ - E +
\delta = F_\mathscr{G}. \label{eqn5.10}\end{equation}
\end{theorem}

\begin{proof}
Suppose there are $E_\mathscr{G}$ edges in $\mathscr{G}$. According
to the definition of the CVR graph, it follows that there are
$\delta$ disconnected parts in $\mathscr{G}$. Hence with the Euler
formula, one gets
$$ F_\mathscr{G} - E_\mathscr{G} + V^+ = \delta. $$
Hence in order to prove Equation \eqref{eqn5.10}, one just needs to
show that $2 V^+ = E + E_\mathscr{G}$. In fact, consider any l-edge
$\ell_i$ with $V^+_i$ crossing vertices. Then $\ell_i$ generates
$V^+_i-1$ edges in $\mathscr{G}$. After running through all the
l-edges in $\mathscr{T}$, each of the crossing vertices is met
twice. Hence it follows that $2V^+ =E + E_\mathscr{G}$, which
finishes the proof of the theorem.
\end{proof}

Theorem \ref{hint} hints us that the basis functions of the
biquadratic spline space over a hierarchical T-mesh could be
constructed around the cells in its CVR graph. Some experiments have
been done on this idea, which will be explored in the future.

Furthermore, we expect that CVR graphs will play an important role
in dimension analysis and basis function construction of higher
degree spline spaces over (hierarchical) T-meshes. For example, we
have the following conjecture:

\begin{conj}
Suppose $\mathscr{T}$ is a hierarchical T-mesh, and its CVR graph is
$\mathscr{G}$. As $m,n \geqslant 2$, it follows that
$$ \dim \overline{\mathbf{S}}(m,n,m-1,n-1,\mathscr{T}) = \dim
\overline{\mathbf{S}}(m-2,n-2,m-3,n-3,\mathscr{G}), $$ where the
spline space $\overline{\mathbf{S}}(m,n,\alpha,\beta,\mathscr{G})$
is defined in a similar way with the spline space over a T-mesh.
\end{conj}

Theorem \ref{hint} states that the conjecture holds as $m=n=2$. As
for $m=n=3$, we have tried many examples, which support this
conjecture as well.

\section{Conclusions}

In this paper the dimension of bilinear and biquadratic spline
spaces over T-meshes are discussed. The basis strategy is by linear
space embedding with an operator of mixed partial derivative. We
obtained the dimension formula of bilinear spline spaces over
general T-meshes, and that of biquadratic spline spaces over
hierarchical T-meshes. Only a lower bound of the dimension is build
for biquadratic spline spaces over general T-meshes.

In the future, the basis function construction of biquadratic
splines spaces and the proposed conjecture will be explored.

\end{document}